\documentclass[english,11pt]{amsart}
\pdfoutput=1
\usepackage[utf8]{inputenc}
\usepackage{babel}
\usepackage{amsmath,amssymb}
\usepackage{geometry}
\usepackage{graphicx}
\usepackage{stmaryrd}
 \usepackage{float}
 \usepackage{units}
\usepackage{stmaryrd}
\usepackage{graphicx}
\usepackage{esint}
\usepackage{subfig}
\usepackage{bbm}
\usepackage{tabularx}
\usepackage{mathtools}
\usepackage{tabularx}
\usepackage{color}
\usepackage{appendix} 
\usepackage{amsthm}
\usepackage{thmtools}
\usepackage{thm-restate} 
\makeatletter

\newtheorem{definition}{Definition}
\newtheorem{theorem}{Theorem}[section]
\newtheorem{lemma}{Lemma}[section]
\newtheorem{proposition}{Proposition}[section]
\newtheorem{coro}{Corollary}[section]

\newtheorem{remark}{Remark}[section]

\newcommand \R {\mathbb{R}}
\newcommand \E {\mathbb{E}}
\newcommand \PP {\mathcal{P}}

\newcommand \Interval[1] {\llbracket 1,#1 \rrbracket}

\newcommand{\Rd}{\mathbb{R}^{d}}

\newcommand{\centereqn}[1]{\hspace{.5\linewidth}\hbox to 0pt{$\displaystyle\hss#1$\hss}}

\DeclareMathOperator*{\argmin}{arg\,min}

\usepackage{algorithm,algpseudocode}

\@ifundefined{showcaptionsetup}{}{%
 \PassOptionsToPackage{caption=false}{subfig}}
\usepackage{subfig}
\makeatother

\usepackage{babel}
\makeatletter
\addto\extrasfrench{%
   \providecommand{\fg}{\ifdim\lastskip>\z@\unskip\fi~\frqq}%
}

  \addto\captionsenglish{}
  \addto\captionsenglish{}
  \addto\captionsenglish{}
  \addto\captionsenglish{}
  \addto\captionsenglish{}
  \addto\captionsfrench{}
  \addto\captionsfrench{}
  \addto\captionsfrench{}
  \addto\captionsfrench{}
  \addto\captionsfrench{}

\addto\captionsfrench{}

\begin{document}

\title[Generalized Wasserstein barycenters]{Generalized Wasserstein barycenters between probability measures living on different subspaces}


\author{Julie Delon}
\address{Universit\'e de Paris, CNRS, MAP5 UMR 8145, F-75006 Paris, France \& Institut Universitaire de France (IUF)}
\email{julie.delon@u-paris.fr}

\author{Nathaël Gozlan \and Alexandre Saint-Dizier}
\address{Universit\'e de Paris, CNRS, MAP5 UMR 8145, F-75006 Paris, France}
\email{ nathael.gozlan@u-paris.fr, alexandre.saint-dizier@u-paris.fr}
 
\date{\today}

\maketitle



\interfootnotelinepenalty=10000 

\begin{abstract}
In this paper, we introduce a generalization of the Wasserstein barycenter, to a case where the initial probability measures live on different subspaces of $\R^d$. We study the existence and uniqueness of this barycenter, we show how it is related to a larger multi-marginal optimal transport problem, and we propose a dual formulation. Finally, we explain how to compute numerically this generalized barycenter on discrete distributions, and we propose an explicit solution for  Gaussian distributions. 
\end{abstract}

\section{Introduction}

In recent years, optimal transport~\cite{villani08} has received a lot
of attention and has become an essential tool to compare or
interpolate between probability distributions. The apparition of
efficient numerical approaches has made optimal transport particularly
successful in numerous applied fields such as economy~\cite{galichon2018optimal}, image 
processing~\cite{delon2004midway, galerne2018texture,
  delon2020wasserstein}, computer vision~\cite{haker2004optimal},
astrophysics~\cite{frisch2002reconstruction, levy2020fast}, 
machine learning~\cite{arjovsky2017wasserstein,genevay2019entropy} and computer
graphics~\cite{levy:hal-01717967}, to name just
a few (for a recent monograph on numerical optimal transport, see~\cite{peyre19book}).

An important tool derived from optimal transport is the notion of Wasserstein barycenter introduced by Agueh and Carlier in \cite{agueh11} (see also \cite{Alvarez-Esteban:2015aa}). In the Euclidean case, the barycenter of
$x_1,\dots,x_p$ with weights $\lambda_1,\dots,\lambda_p$ (positive and
summing to 1) is the point $x$ of $\R^d$ which minimizes $\sum_{i=1}
^p\lambda_i |x-x_i|^2$, where  $|\,.\,|$ denotes the Euclidean norm on $\Rd$.  The Wasserstein barycenter is obtained in the same way in the space $\mathcal{P}_2(\R^d)$ of probability measures with second order moments, by replacing the Euclidean distance by the square Wasserstein distance $W_2$. 

In this paper, we propose a generalization of the notion of
Wasserstein barycenter, to a case where the considered probability
measures live on different subspaces of $\R^d$. Relying on the same
Euclidean analogy as above, for $p$ vectors  $x_i \in \R^{d_i}$ and
$p$ linear transformations $P_i:  \R^d \rightarrow \R^{d_i}$,
$i=1,\dots,p$, a generalized barycenter between these $x_i$ can be
defined as a minimizer in $\R^d$ of  $\sum_{i=1} ^p\lambda_i
|P_i(x)-x_i|^2$. A solution is given by $\hat{x}= (\sum_{i=1}^p\lambda_i P_i^T P_i)^{-1}(\sum \lambda_{i=1}^p P_i^T x_i)$ when the matrix
$\sum_{i=1}^p \lambda_i P_i^T P_i$ is full rank. Our generalized
Wasserstein barycenter is obtained by replacing the vectors $x_i$ by
$p$ probability measures $\nu_i$ on their respective subspace
$\R^{d_i}$ and the Euclidean distance by $W_2$. In other words, we 
study the minimization problem
\begin{equation}
\inf_{\gamma \in \mathcal{P}_2(\R^d)} \sum_{i=1}^p \lambda_i W_2^2(P_i\#\gamma,\nu_i),\label{eq:projection_consensus}
\end{equation}
where $P_i\#\gamma$ denotes the push-forward of $\gamma$ by $P_i$, {\em i.e.} the measure on $\R^{d_i}$ such that
  $\forall A \subset\R^{d_i} $, $(P_i\#\gamma)(A) =
  \gamma(P_i^{-1}(A))$. Figure~\ref{fig:belle_figure} illustrates this
  notion on an 3D example with four projections on different planes. 
 A solution $\gamma$
of~\eqref{eq:projection_consensus} realizes a  consensus between all
the $\nu_i$ for the Wasserstein distance $W_2$, through the
transformations $P_i$. 
Observe that in most cases, the problem will have an infinity of
solutions, since any measure $\mu$ such that $P_i\# \mu  = P_i \# \gamma$
for all $i$
will also be solution. For instance, if $\R^d = \R^{d_1}\times \dots
\times \R^{d_p}$ and the $P_i$ are canonical projections on the
subspaces $ \R^{d_i}$, any $\mu$
with marginals $\nu_1,\dots,\nu_p$ will be solution. 
This formulation generalizes the classical notion of Wasserstein
barycenter~\cite{agueh11}, obtained when all the $P_i$ are  equal to
$\mathrm{I}_d$, the
identity application on $\R^d$.

\begin{figure}
  \centering
  \includegraphics[width=0.5\textwidth]{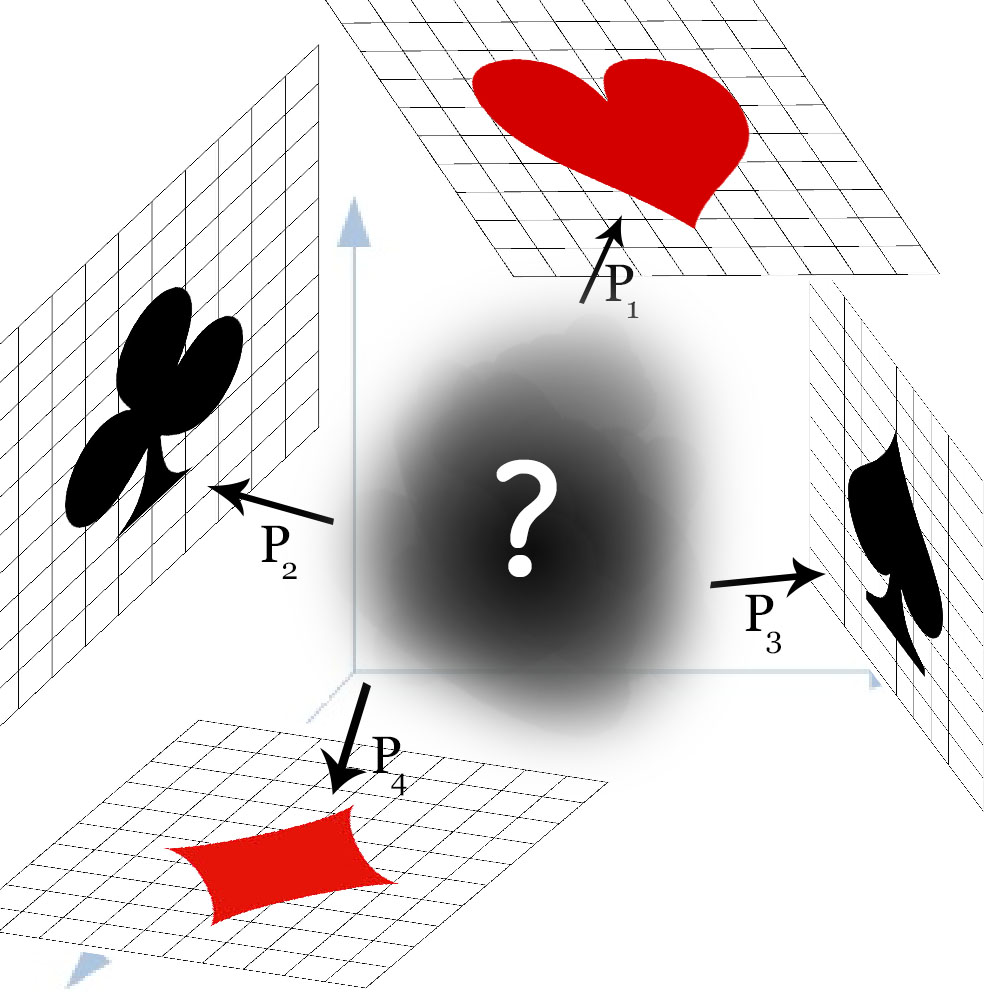}
  \caption{Generalized barycenter.  In this example, the
    transformations $P_i:\R^3 \rightarrow \R^2$ are linear projections
    from $\R^3$ on four different plans. Knowing the four probability
    measures $\nu_1,\dots,\nu_4$ on these four plans, we look for
the 3d probability measure $\gamma$ which realizes a consensus between
these four projections, in the sense described by Equation~\eqref{eq:projection_consensus}.}
  \label{fig:belle_figure}
\end{figure}

Gaspard Monge, one of the founding fathers of optimal transport, is
also the father of descriptive geometry~\cite{monge1798geometrie}, and one of the goals of this
discipline is to represent information about a three
dimensional volume  or surface from several well chosen two
dimensional projections. An obvious application of the generalized
barycenter is precisely the reconstruction of a
measure in $d$ dimensions from the knowledge of projections of this measure on
different subspaces. In practice, these
projections can be noisy or contain errors, and therefore do not
necessarily coincide on their common subspaces.

A concrete application where a distribution must be reconstructed from a set of marginals appears in image processing with patch-based aggregation~\cite{saint20}. Patches are small overlapping image pieces and it is usual to infer stochastic models (for example a Gaussian or GMM distribution) on these patches~\cite{delon2018gaussian}. Typically, each of these models is a distribution in $\R^9$  (for $3\times 3$ patches). The aggregation problem consists in reconstructing a distribution on the whole image (hence in $\R^d$ with $d$ very large) from the knowledge of all these overlapping (and generally not coinciding) models on patches.  

More generally, the reconstruction of a multidimensional distribution from a finite set of projections appears in many  applied fields,  for instance in medical or geophysical imaging~\cite{natterer2001mathematical}. In  these applications, a system measures multiple 1D or 2D projections of a 3D density and the objective is to reconstruct the 3D signal from these projections. The Fourier transform of these projections can be shown to be slices of the Fourier transform of the complete density. If the knowledge of all the slices is theoretically necessary to reconstruct this Fourier transform, in practical applications only a finite number of these slices is known, and they are interpolated to reconstruct the Fourier transform and recover the complete density by Fourier inversion.  

A particular instance of Problem~\eqref{eq:projection_consensus}, where
all but one the $P_i$ are one dimensional
projections, and the last one is equal to $\mathrm{I}_d$, has been studied
in~\cite{abraham17} and also considered in~\cite{benamou15}, precisely for tomographic reconstruction.  The authors derive
a dual formulation for the problem and prove existence and uniqueness of the solution. Several of these results rely on the  presence of a fully dimensional prior (one of the $P_i$ is invertible). We do
not make this assumption in this paper and study the problem in its full generality, without assumption on the linear applications $P_i$.  One of the obvious consequences is that in most cases, 
we lose uniqueness of the solution of~\eqref{eq:projection_consensus}, but we show that most properties are preserved.

The contributions of the paper are the following. After a short reminder on Wasserstein distances and Wasserstein barycenters in Section~\ref{sec:background}, we show in Section \ref{sec: existence unicity} the existence of solutions for the minimization problem~\eqref{eq:projection_consensus} and we explain how it is related to a multi-marginal optimal transport problem
in dimension $\sum_{i=1}^p d_i$. To this aim, we propose a convenient reformulation of the generalized barycenter problem as a classical Wasserstein barycenter problem between degenerate distributions $\tilde{\nu}_i$ which are obtained from the $\nu_i$.    We then propose in Section~\ref{sec: dual} a dual formulation for~\eqref{eq:projection_consensus} (which contains the of~\cite{abraham17} as a particular case). In Section~\ref{sec:gaussian_gwb}, we show that when the distributions $\nu_i$ are Gaussian, there is at least one Gaussian solution to~\eqref{eq:projection_consensus}, we study the uniqueness of this Gaussian solution, and how to reconstruct it in practice. Finally, we explain in Section \ref{sec: experiments} how to compute this generalized barycenter numerically and we provide  several numerical experiments to illustrate its behavior on discrete and Gaussian distributions.

\section{A short reminder on Wasserstein barycenters and multi-marginal optimal transport}

\label{sec:background}
Let $\mathcal{P}_2(\R^d)$ be the set of probability measures on $\R^d$
with finite second order moment.  For two probability measures $\nu_1, \nu_2$  in $\mathcal{P}_2(\R^d)$, the Wasserstein distance $W_2(\nu_1,\nu_2)$ is defined as  

\begin{equation}
  \label{eq:wasserstein_definition}
   W_2^2(\nu_1,\nu_2) :=  \inf_{X_1\sim \nu_1; X_2\sim \nu_2} \E \left(|X_2 - X_1|^2\right) =  \inf_{\pi \in \Pi(\nu_1,\nu_2)} \int_{\R^d\times\R^d} |x_2-x_1|^2\,d\pi(x_1,x_2),
 \end{equation}
 where  $\Pi(\nu_1,\nu_2)\subset \mathcal{P}_2(\R^d\times \R^d)$  is the subset of probability distributions $\pi$ on $\R^d\times \R^d$ with marginal distributions $\nu_1$ and $\nu_2$.

It is well known that there always exists a couple $(X_1,X_2)$ of random variables attaining the infimum in \eqref{eq:wasserstein_definition} (see e.g.~\cite[Chapter 4]{villani08}). This
 couple is named \textit{optimal coupling} and its distribution
 $\pi$ is called an \textit{optimal transport plan} between $\nu_1$ and $\nu_2$. This plan distributes all the mass of the distribution $\nu_1$ onto the distribution $\nu_2$ with a minimal cost, and the quantity $W_2^2(\nu_1,\nu_2)$ is the corresponding total cost. It is also well known that $W_2$ defines a metric on $\mathcal{P}_2(\R^d)$ (see e.g.~\cite[Chapter 6]{villani08}).

The barycenter of $x_1,\dots,x_p \in \R^d$, for positive weights $\lambda_1,\dots,\lambda_p$
summing to 1, is defined as 
\begin{equation}
B(x_1,\dots,x_p) = \sum_{i=1}^{p}\lambda_i x_i = \mathrm{argmin}_{y\in\R^d}
\sum_{i=1}^{p}\lambda_i | x_i -y|^2.
\label{eq:barycenter}
\end{equation}
Similarly, the Wasserstein barycenter of $p$ probability measures
$\nu_1,\dots, \nu_p$ for the positive weights $\lambda_1,\dots,\lambda_p$ is a solution of the minimization problem
 \begin{equation}
    \label{eq:W2_bary}
    \inf_{\nu} \sum_{i=1}^{p} \lambda_i W_2^2(\nu_i,\nu).
  \end{equation}
If all weights are equal to $\frac 1 p$, a Wasserstein barycenter is also a Fr\'echet mean for the Wasserstein distance $W_2$ (see~\cite{peyre19book}).

Agueh and Carlier have studied in depth the questions of existence and
uniqueness of barycenters for $W_2$ in~\cite{agueh11}. They show in particular that the solutions of this
barycenter problem are related to the solutions of another
optimization problem, called the multi-marginal transport problem~\cite{Gangbo1998}, which can be written

 \begin{equation} \label{eq:W2_multimarge_definition}
 \inf_{\pi\in\Pi(\nu_1,\nu_2,\dots,\nu_{p})} \int_{\R^d\times \dots \times \R^d}
      \sum_{i=1}^{p} | x_i-B(x_1,\ldots,x_p)|^2\,
      d\pi(x_1,x_2,\ldots,x_{p}), 
\end{equation}
where $\Pi(\nu_1,\nu_2,\dots,\nu_{p})$ is the set of probability
measures on $(\R^d)^{p}$ with $\nu_1,\nu_2,\dots,\nu_{p}$ as
marginals.  The following proposition summarizes some of the main contributions of \cite{agueh11}.

\begin{proposition}[Agueh-Carlier \cite{agueh11}]\ \label{prop:MM-AC}
\begin{enumerate}
\item Both Problems \eqref{eq:W2_bary} and \eqref{eq:W2_multimarge_definition} admit solutions. 
\item If at least one of the probability measures $\nu_i$ has a density with respect to Lebesgue, then Problem \eqref{eq:W2_bary} admits a unique solution. 
\item If $\pi^*$ is a solution of ~\eqref{eq:W2_multimarge_definition},
then $\nu^* = B\# \pi^*$ is a solution of~\eqref{eq:W2_bary}, and
the infimum of~\eqref{eq:W2_multimarge_definition} and~\eqref{eq:W2_bary} are equal. 
\end{enumerate}
\end{proposition}
Other contributions of \cite{agueh11}, in particular concerning the dual problem of~\eqref{eq:W2_bary}, will be recalled below.

 \section{Generalized Wasserstein barycenters}
 \label{sec: existence unicity}
 
In the Wasserstein barycenter problem~\eqref{eq:W2_bary}, the measures
$\nu_1,\dots,\nu_p$ are seen as living in the same space $\R^d$. The generalized Wasserstein
barycenters can be seen as a variant of this problem, where
the spaces in which the $\nu_i$ live are subspaces of a larger common space.

 \begin{definition}[GWB]
 Given $p$ positive integers $d_1,\dots,d_p$, $p$  probability
 measures $(\nu_1,\ldots,\nu_p)\in\PP_2(\R^{d_1})\times\dots\times
 \PP_2(\R^{d_p})$, a vector $\lambda = (\lambda_1,\ldots,\lambda_p)$
 of positive weights summing to 1 and $p$ linear
 applications  
 $P_i:\R^d\rightarrow \R^{d_i}$, define the functional
\begin{equation}\label{eq:F}
\mathcal{F}(\gamma):=\sum_{i=1}^p \lambda_iW_2^2(\nu_i,P_i\#\gamma), \qquad \gamma \in \mathcal{P}_2(\R^d).
\end{equation}
A solution $\gamma^*\in \mathcal{P}_2(\R^d)$
 of the minimization problem
 \begin{equation}
 \inf_{\gamma \in \mathcal{P}_2(\R^d)} \mathcal{F}(\gamma)
 \tag{GWB}
\label{ref:GWB}
\end{equation}
is called a generalized Wasserstein barycenter of the marginals $\nu_i$ for
the applications $P_i$.
 \end{definition}
Note that this formulation includes the classical notion of Wasserstein barycenter, by taking $d_i=d$ and $P_i=\mathrm{I}_d$, for all $i\in \llbracket 1,p \rrbracket$. Observe also that \eqref{ref:GWB} is a convex minimization problem. In the numerical experiments, we will always assume that for all $i\in \llbracket 1,p \rrbracket$, $d_i \leq d$ and $P_i$ is surjective. 

In Section \ref{sec:reformulation} below, we will relate \eqref{ref:GWB} to the classical Wasserstein barycenter problem. This will in particular enable us to show existence of solutions for \eqref{ref:GWB}.

\subsection{A technical reduction}\label{sec:technical}
In all what follows, we will denote by $A$ the $d\times d$ symmetric semidefinite matrix defined by
\[
A = \sum_{i=1}^p \lambda_i P_i^TP_i
\]
(here, and in all the paper, we make no difference between linear maps and their matrices in the standard basis of $\R^d$ and $\R^{d_i}$).

As we will see, it will be very convenient to assume that the matrix $A$ is invertible (or equivalently positive definite). In the sequel, we explain that this is not a real restriction and that it is always possible to reduce to this situation.

\begin{lemma}
The matrix $A$ is positive definite if and only if 
\begin{equation}\label{eq:matrixA}
\bigcap_{i=1}^p \mathrm{Ker}(P_i)= \{0\}.
\end{equation}
\end{lemma}
\proof
Observe that, for all $x \in \R^d$,
\[
x^TAx =  \sum_{i=1}^p \lambda_ix^T P_i^TP_ix =  \sum_{i=1}^p \lambda_i |P_ix|^2
\]
and so $x^TAx\geq0$ with equality if and only if $x\in \bigcap_{i=1}^p \mathrm{Ker}(P_i)$.
\endproof

\begin{lemma}
Denote by $K = \bigcap_{i=1}^p \mathrm{Ker}(P_i)$ and by $P_{K^\perp}$ the orthogonal projection on $K^\perp$. 
A probability measure $\nu \in \mathcal{P}_2(\R^d)$ is solution to \eqref{ref:GWB} if and only if $P_{K^\perp}\# \nu$ is solution to \eqref{ref:GWB}.
\end{lemma}
\proof
This comes from the fact, if $X$ has law $\nu$, then for all $i \in \llbracket 1,p \rrbracket$,
\[
P_iX = P_iP_{K^\perp}X,
\]
and so $P_iX$ and $P_iP_{K^\perp}X$ have the same law.
\endproof

Thanks to the preceding lemma, we see that solutions of \eqref{ref:GWB} are the laws of random vectors $X$ of the form $X=X_1+X_2$, where $X_1$ takes values in $K^\perp$, $X_2$ in $K$ and the law $\nu_1$ of $X_1$ is solution of \eqref{ref:GWB}. We have thus reduced the problem to determine solutions of \eqref{ref:GWB} supported on $K^\perp$. For that purpose, let $\bar{d}$ be the dimension of $K^\perp$ and let $e_1,\ldots, e_{\bar{d}}$ be a basis of $K^\perp$. For all $i \in \llbracket 1,p \rrbracket$, define 
\[
\bar{P}_i = P_iQ,
\]
where $Q$ is the $d \times \bar{d}$ matrix of $e_1,\ldots, e_{\bar{d}}$ in the standard basis of $\R^d$.
If $\bar{\gamma}^*$ is a minimizer of the functional $\bar{\mathcal{F}}$ defined by
\begin{equation}\label{eq:barF}
\bar{\mathcal{F}}:=\sum_{i=1}^p
 \lambda_iW_2^2(\nu_i,\bar{P}_i\#\bar{\gamma}),\qquad \bar{\gamma} \in \mathcal{P}_2(\R^{\bar{d}}),
\end{equation}
then $Q\#\bar{\gamma}^*$ is a solution of \eqref{ref:GWB} supported on $K^\perp$. Finally, observe that 
\begin{equation}\label{eq:barA}
\bar{A} = \sum_{i=1}^p \lambda_i \bar{P}_i^T\bar{P}_i
\end{equation}
is now invertible.

\subsection{Reformulation of the generalized Wasserstein barycenter problem}
\label{sec:reformulation}
In this section, we assume that the matrix $A$ defined in \eqref{eq:matrixA} is invertible. The goal of this section is to relate the generalized Wasserstein barycenter problem to the usual one. 

Recall the definition of the functional $\mathcal{F}$ given in \eqref{eq:F} and let us introduce another functional $\mathcal{G} : \mathcal{P}_2(\R^d) \to \R_+$ defined as follows
\begin{equation}
\label{eq:G}
\mathcal{G}(\mu) = \sum_{i=1}^p \lambda_i W_2^2(\tilde{\nu}_i, \mu),\qquad \mu \in  \mathcal{P}_2(\R^d),
\end{equation}
where, for all $i\in \llbracket 1,p \rrbracket$,
\[
\tilde{\nu}_i = (A^{-1/2}P_i^T)\# \nu_i.
\]
Note that the matrix $A$ being symmetric and definite positive, it admits a unique square root denoted by $A^{1/2}$ which is itself invertible. 

The following proposition will be used extensively in all the paper.
\begin{proposition}\label{prop:lien}
A probability measure $\gamma^* \in \mathcal{P}_2(\R^d)$ minimizes $\mathcal{F}$ if and only if  $\mu^* = A^{1/2}\#\gamma^*$ minimizes $\mathcal{G}$.
\end{proposition}
In other words, $\gamma^*$ is a generalized Wasserstein barycenter of the $\nu_i$ for the applications $P_i$ and weights $\lambda_i$ if and only if $\gamma^* = A^{-1/2}\# \mu^*$, where $\mu^*$ is the Wasserstein barycenter of the measures $\tilde{\nu}_i$ with weights $\lambda_i$. Proposition \ref{prop:lien} will enable us to extend to the case of the generalized Wasserstein barycenter problem \eqref{ref:GWB} many properties which are known in the classical case. We will in particular use Proposition \ref{prop:lien} to obtain a dual formulation of \eqref{ref:GWB} (see Section \ref{sec: dual}) and to study the case where the $\nu_i$ are Gaussian (see Section \ref{sec:gaussian_gwb}).

\proof Let $\gamma \in \mathcal{P}_2(\R^d)$ ; setting $\mu = A^{1/2}\#\gamma$, it follows from Lemma \ref{lem:couplage} below that
\begin{align*}
W_2^2(\nu_i, P_i\# \gamma)&= \int |x|^2 \,d\nu_i(x) + \int y^TP_i^TP_iy\, d\gamma(y) - 2\sup_{\tilde{\pi} \in \Pi(\nu_i, P_i\# \gamma)} \int x\cdot z \,d\tilde{\pi}(x,z)\\
& = \int |x|^2 d\nu_i(x) + \int y^TP_i^TP_iy\, d\gamma(y) - 2\sup_{\pi \in \Pi(\nu_i,\gamma)} \int x\cdot P_iy \,d\pi(x,y)\\
&  = \int |x|^2 \,d\nu_i(x) + \int y^TP_i^TP_iy\, d\gamma(y) - 2\sup_{\pi \in \Pi(\nu_i,\gamma)} \int P_i^Tx\cdot y \,d\pi(x,y)\\
&  = \int |x|^2 \,d\nu_i(x) + \int y^T A^{-1/2}P_i^TP_iA^{-1/2}y\, d\mu(y) - 2\sup_{\pi \in \Pi(\nu_i,\gamma)} \int A^{-1/2}P_i^Tx\cdot A^{1/2}y \,d\pi(x,y)\\
&  = \int |x|^2 \,d\nu_i(x) + \int y^T A^{-1/2}P_i^TP_iA^{-1/2}y\, d\mu(y) - 2\sup_{\pi \in \Pi(\tilde{\nu}_i,\mu)} \int x\cdot y \,d\pi(x,y)\\
& =  W_2^2(\tilde{\nu}_i, \mu) + \int |x|^2- x^TP_iA^{-1}P_i^Tx \,d\nu_i(x) + \int y^T A^{-1/2}P_i^TP_iA^{-1/2}y- |y|^2\, d\mu(y).
\end{align*}

Therefore, summing these equalities and using \eqref{eq:matrixA} gives
\[
\mathcal{F}(\gamma) = C + \mathcal{G}(\mu),
\]
with $C =  \sum_{i=1}^p \lambda_i \int |x|^2- x^TP_iA^{-1}P_i^Tx \,d\nu_i(x).$
\endproof
In the proof of Proposition \ref{prop:lien} we have used the following result.

\begin{lemma}\label{lem:couplage}
Let $\gamma \in \mathcal{P}_2(\R^d)$, $\nu \in \mathcal{P}_2(\R^{d'})$ and $T: \R^d \to \R^{d'}$ be a measurable map. Then
$\pi' \in \Pi(\nu, T\#\gamma)$ if and only if there is some $\pi \in \Pi(\nu,\gamma)$ such that $\pi' = (\mathrm{I}_d,T)\#\pi$.\\
In particular, if there exist $a,b\geq0$ such that $|T(x)| \leq a+b|x|$ for all $x \in \R^d$, then
\[
\sup_{\pi \in \Pi(\nu,\gamma)} \int x\cdot T(y) \,d\pi(x,y) = \sup_{\pi \in \Pi(\nu,T\#\gamma)} \int x\cdot z \,d\pi(x,z).
\]
\end{lemma}

\proof
If $\pi \in \Pi(\nu,\gamma)$ and $\pi'$ is the image of $\pi$ under the application $(x,y)\mapsto (x,T(y))$, then it is clear that $\pi' \in \Pi(\nu,T\#\gamma)$.

Conversely, let us prove that all couplings between $\nu$ and $T\#\gamma$ is of this form. 
Let $Y \sim \gamma$ and consider $P = \mathrm{Law}(Y,T(Y))$. 
The coupling $P$ admits the following disintegration 
\[
dP(y,z) = d\gamma'(z)dp^z(y),
\]
where  $\gamma' = T\#\gamma$ is the second marginal of $P$ and $(p^z)_{z\in \R^{d'}}$ is a probability kernel from $\R^{d'}$ to $\R^d$.
Let $\pi' \in \Pi(\nu,T\#\gamma)$ and define
\[
d\pi(x,y) =  d\pi'(x,z)dp^z(y).
\]
In other words, for all positive measurable function $f$
\[
\int f(x,y) \,d\pi(x,y) = \int f(x,y)\,dp^z(y)d\pi'(x,z).
\]
Therefore
\[
\int g(x) \,d\pi(x,y) = \int g(x)\,dp^z(y)d\pi'(x,z) = \int g(x)\,d\pi'(x,y) = \int g(x)\,d\nu(x)
\]
and
\[
\int g(y) \,d\pi(x,y) = \int g(y)\,dp^z(y)d\pi'(x,z) = \int g(y)\,dp^z(y)d\gamma'(z)  = \int g(y)\,dP(y,z) = \int g(y)\,d\gamma(y),
\]
which proves that $\pi \in \Pi(\nu,\gamma)$.
Furthermore, since $d\pi'(x,z) = d\gamma'(z) d\pi'_z(x)$, it holds
\begin{align*}
 \int f(x, T(y)) \,d\pi(x,y) &= \int f(x, T(y))\,dp^z(y)d\pi'(x,z)\\
 & =  \int f(x, T(y))\,dp^z(y)d\gamma'(z) d\pi'_z(x)\\
 & = \int f(x, z)\,d\gamma'(z) d\pi'_z(x) \\
 & =  \int f(x,z) \,d\pi'(x,z),
\end{align*}
where the third equality comes from the fact that for $\gamma'$ almost every $z$, $T(y)=z$ for $p^z$ almost all $y$.
One concludes that $\pi'= (\mathrm{I}_d,T)\#\pi$, which completes the proof.
\endproof

\subsection{Existence of solutions for (GWB)}
We show in the following that the primal minimization problem~\eqref{ref:GWB} has solutions and that
a solution is generally not unique. 

\begin{proposition}
The problem~\eqref{ref:GWB} has solutions.
More precisely, the function $\mathcal{F}$ defined in \eqref{eq:F} admits at least one minimizer $\gamma^* \in \mathcal{P}_2(\R^d)$. Moreover, if at least one of the measures $\tilde{\nu}_i$ is absolutely continuous the minimizer is unique. 
\end{proposition}
\proof 
As explained in Section \ref{sec:technical}, one can assume without loss of generality that the matrix $A$ defined in \eqref{eq:matrixA} is invertible. According to Proposition \ref{prop:lien}, it is enough to show that $\mathcal{G}$ attains its minimum. This follows from \cite[Proposition 2.3]{agueh11}. Let us briefly recall the argument. The functional $\mathcal{G}$ is lower semi-continuous (for the usual weak topology) and its sublevel sets $\{\mathcal{G}\leq k\}$, $k\geq0$, are compact (thanks to Prokhorov Theorem), thus it attains its minimum. Uniqueness follows from \cite[Proposition 3.5]{agueh11}. It can also be derived from the strict convexity of $\mu \mapsto W_2^2(\nu,\mu)$, when $\nu$ has a density (\cite[Corollary 2.10]{ABCM11}).
\endproof

Let us emphasize that we don't have uniqueness of the solution
in general. Even if  $A=\sum_{i=1}^p \lambda_i P_i^T P_i$ is invertible, if $\gamma^*$ is a solution, any probability distribution $\mu$
on $\R^d$ satisfying $P_i \# \gamma^*  = P_i \# \mu$ for all $i=1,\dots,
p$ is also a solution of the minimization problem for instance. The question of the
existence and uniqueness of probability measures with known and overlapping absolutely continuous marginals is an important problem in probability, see for instance the recent~\cite{kazi2019construction}.

The following proposition shows that in certain circumstances, the projections of generalized barycenters are uniquely determined.

\begin{proposition}\label{prop:uniquenessproj}
Suppose that $\nu_i$ has a density on $\R^{d_i}$ for all $i \in \llbracket 1,p \rrbracket$, then if $\gamma_1$ and $\gamma_2$ are solutions of \eqref{ref:GWB}, it holds $P_i\#\gamma_1=P_i\#\gamma_2$ for all $i \in \llbracket 1,p \rrbracket$.
\end{proposition}
\proof
Let $\gamma_1$ and $\gamma_2$ be two solutions of \eqref{ref:GWB}, and let us assume by contradiction that $P_{i_o}\#\gamma_1\neq P_{i_o}\#\gamma_2$ for some $i_o  \in \llbracket 1,p \rrbracket$. Consider the function $f(t) = \mathcal{F}((1-t)\gamma_1 + t\gamma_2)$, $t\in [0,1].$ Since for all $i$, $P_i\# ((1-t)\gamma_1 + t\gamma_2) = (1-t)P_i\# \gamma_1 + tP_i\#\gamma_2$, the function $f$ is convex on $[0,1]$. Moreover, since $\nu_{i_o}$ has a density and $P_{i_o}\#\gamma_1\neq P_{i_o}\#\gamma_2$, it follows from \cite[Corollary 2.10]{ABCM11} that the function $t \mapsto W_2^2 (\nu_{i_o}, (1-t)P_{i_o}\# \gamma_1 + tP_{i_o}\#\gamma_2)$ is strictly convex. Thus $f$ is strictly convex as well and so 
\[
f(1/2) = \mathcal{F}\left( \frac{1}{2}\gamma_1 + \frac{1}{2}\gamma_2 \right) <  \frac{1}{2}f(0) + \frac{1}{2}f(1) = \frac{1}{2}\mathcal{F}\left(\gamma_1 \right) + \frac{1}{2}\mathcal{F}\left(\gamma_2 \right) = \min \mathcal{F},
\]
which contradicts the fact that $\gamma_1$ and $\gamma_2$ are minimizers of $\mathcal{F}$ and completes the proof.
\endproof

\medskip

In the case of discrete measures supported by a finite number of points, we do not have uniqueness for the classical barycenter~\cite{anderes2016discrete} and this also holds true for the generalized barycenter. However, if two solutions  $\gamma_1$ and $\gamma_2$ of~\eqref{ref:GWB} in the discrete case are  such that $P_i \# \gamma_1 = P_i \# \gamma_2$ for all $i\in \llbracket 1,p \rrbracket$, with $p$ large enough, then it can be shown that $\gamma_1 = \gamma_2$ since a finite but large enough number of projections can characterize a point cloud~\cite{heppes1956determination,fernandez2017determining}. 

Figure \ref{fig:no_unicity} shows an example where the measures
$\nu_1,\dots,\nu_p$ are all several 1d projections of a  discrete measure
(in yellow) in $\R^2$. In this discrete case, the problem~\eqref{ref:GWB} has
several solutions, including the yellow distribution for which
the value of the energy is $0$. We show in black the reconstruction of a
probability measure with exactly the same projections (the algorithm used for
this reconstruction will be described in Section~\ref{sec:computing}).
We see that when the number of 1d projections increases, the reconstructed measure
gets closer to the discrete yellow measure. For 5 projections, the
      reconstructed generalized barycenter is the same as the original distribution in this example.

\begin{figure}

    \subfloat{
      \includegraphics[width=7cm, trim={14cm 2cm 13cm 2cm},clip]{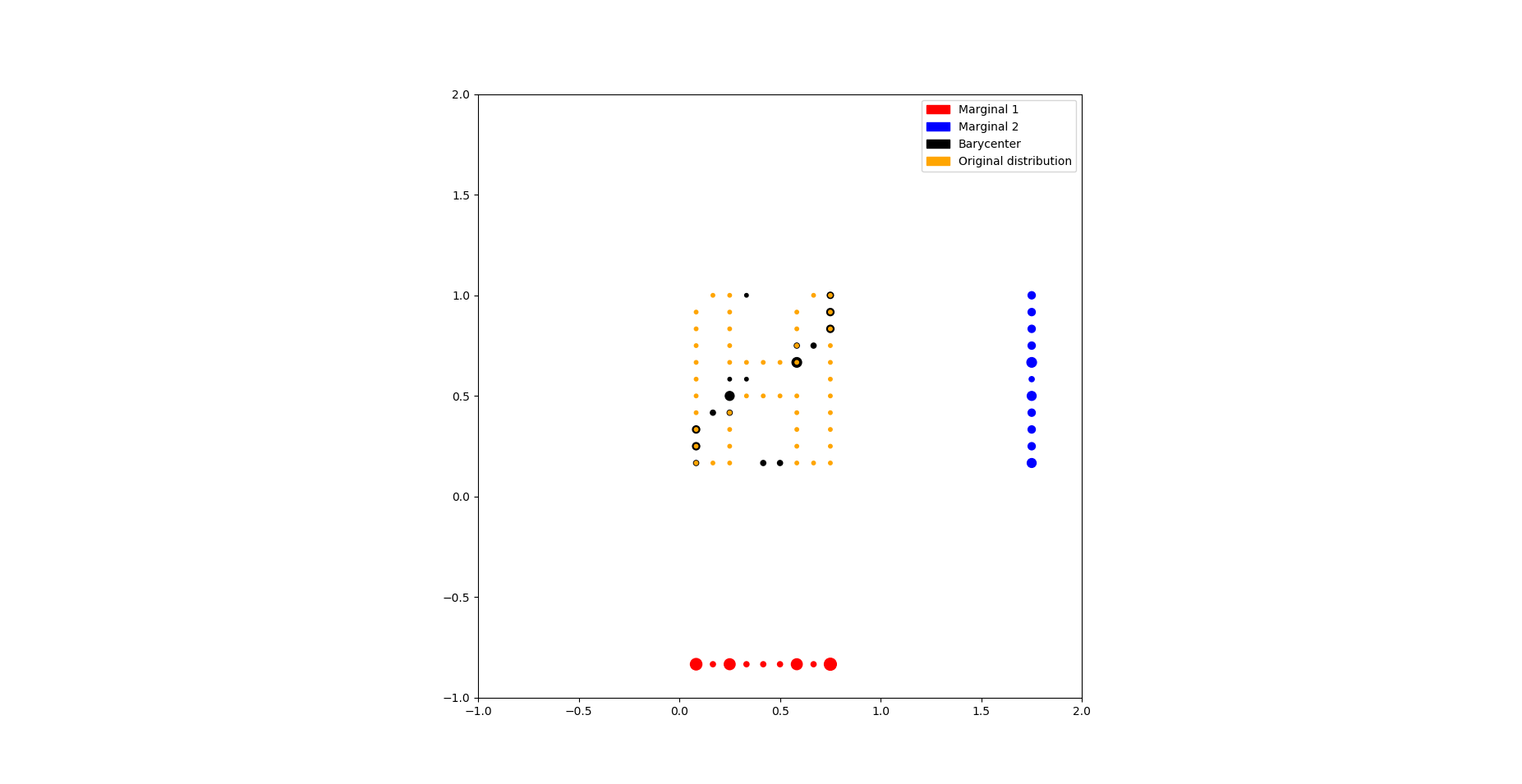}}
     \subfloat{
      \includegraphics[width=7cm, trim={14cm 2cm 13cm 2cm},clip]{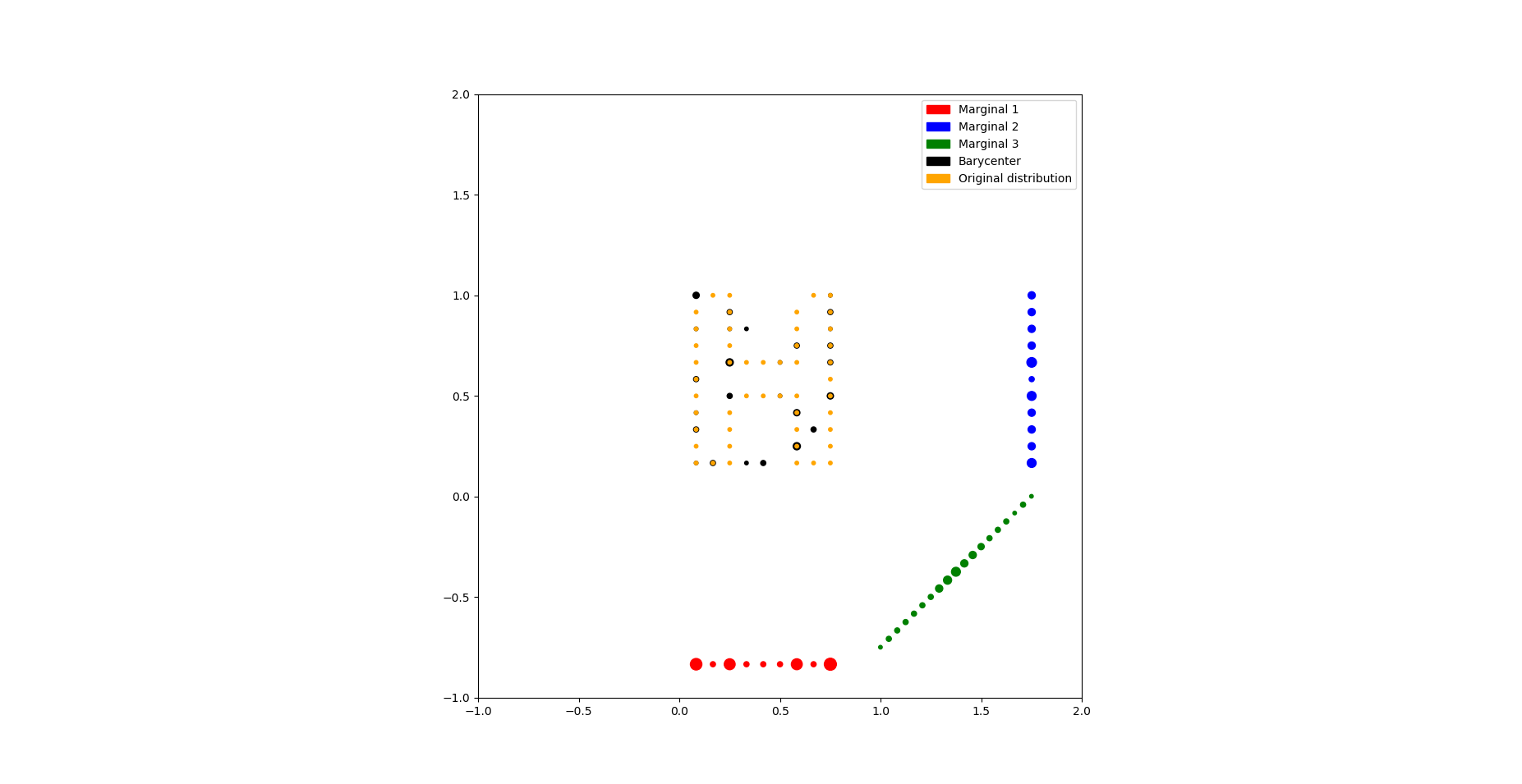}}

    \subfloat{
      \includegraphics[width=7cm, trim={14cm 2cm 13cm 2cm},clip]{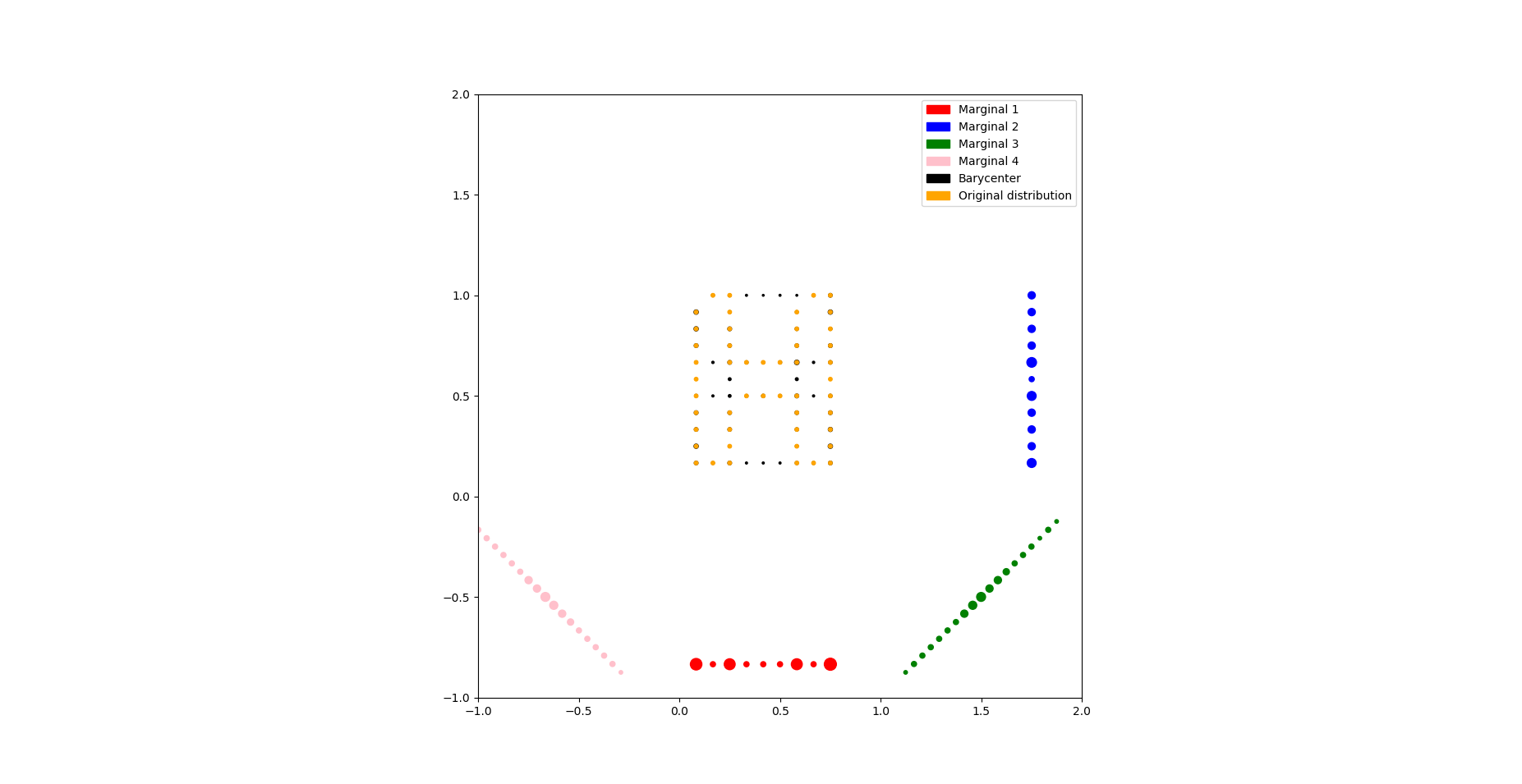}}
      \subfloat{
    \includegraphics[width=7cm, trim={14cm 2cm 13cm 2cm},clip]{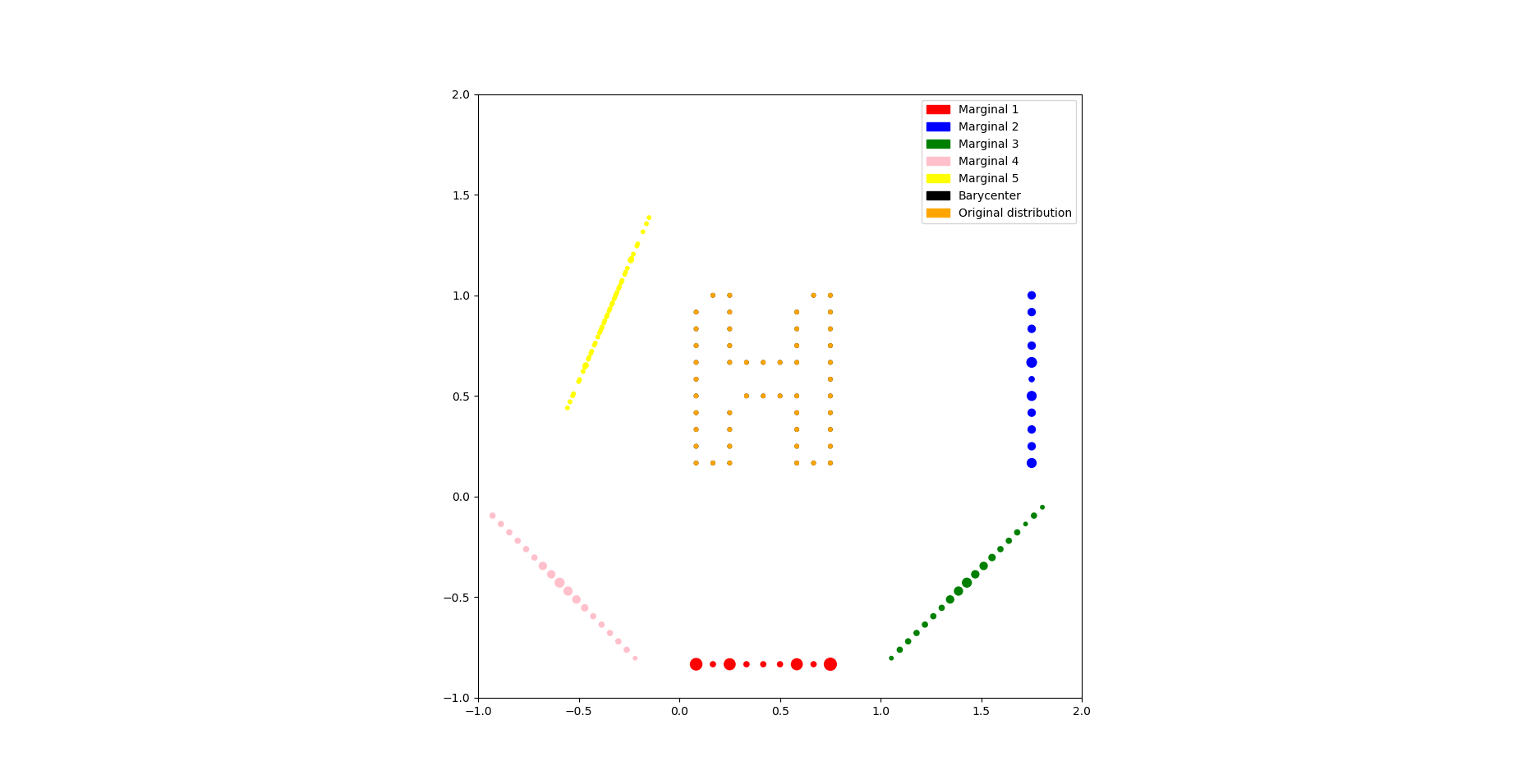}}
      
    \caption{In this example, the measures
$\nu_1,\dots,\nu_p$ are several 1d projections of a  discrete measure
(shown in yellow) in $\R^2$. The number of projections varies from $2$ to
$5$. The reconstructed generalized
      barycenter is shown in black, illustrating the non uniqueness of
      the solution (the black distribution has exactly the same
      projections as the yellow one). The larger the number of marginals, the
      more accurate is the reconstruction. For 5 projections, the
      reconstructed generalized barycenter is the same as the original distribution.}
    \label{fig:no_unicity}
\end{figure}

\subsection{Link between (GWB) and multi marginal optimal transport \label{sec: multimarginal}}
In this section, we assume that the matrix $A$ defined in \eqref{eq:matrixA} is invertible.
For $x=(x_1,\ldots,x_p)\in \R^{d_1}\times\cdots\times \R^{d_p}$,
 we define the generalized Euclidean barycenter as follows
\[
B_{\text{gen}}(x)=B_{\text{gen}}(x_1,\ldots,x_p) =  \arg\min_{y\in \R^d} \sum_{i=1}^p 
 \lambda_i |x_i-P_i(y)|^2 = A^{-1} \sum_{i=1}^p\lambda_i P_i^T(x_i),
\]
 and we also define the cost $c$ by
\[
c(x)=\sum_{i=1}^p \lambda_i |x_i-P_i(B_{\text{gen}}(x))|^2,\qquad x=(x_1,\ldots,x_p)\in \R^{d_1}\times\cdots\times \R^{d_p}.
\]
 We propose to study the multimarginal problem for the measures
 $\nu_1,\ldots,\nu_p$ and cost function $c$, {\em i.e.}
\begin{equation}
 \inf \left\{ \int_{\R^{d_1+\dots+d_p}} c(x)\,d\pi(x),\pi\in
   \Pi(\nu_1,\ldots,\nu_p)\right\}.
 \tag{MM}
 \label{eq:MMOT}
\end{equation}
\begin{proposition}
\label{prop:p=GWB}
The infimum in~\eqref{eq:MMOT} and~\eqref{ref:GWB} are equal. Moreover, if $\pi^*$ is a solution
of~\eqref{eq:MMOT}, then $\gamma^* = B_{\text{gen}}\#\pi^*$  is a solution of~\eqref{ref:GWB}.
\end{proposition}
\begin{proof}
Let us denote $D = d_1+\cdots+d_p$. 
First let us show that $~\eqref{ref:GWB}\leq~\eqref{eq:MMOT}.$
Let $\pi\in\Pi(\nu_1,\ldots,\nu_p)$ and define $\gamma=B_{\text{gen}}\#\pi$. For all $i \in \llbracket 1,p \rrbracket$, we define $p_i$ the projection from $\R^D=\R^{d_1}\times\cdots\times \R^{d_p}$ to $\R^{d_i}$ such that $p_i(x_1,\ldots,x_p)=x_i$, and define $\eta_i = (p_i,P_i\circ B_{\text{gen}})\#\pi$. By construction, $\eta_i \in \Pi(\nu_i,P_i\#\gamma)$. Thus, 
\[
W_2^2(\nu_i,P_i\#\gamma)\leq \int_{\R^{d_i}\times \R^{d_i}} |x-y|^2\,d\eta_i(x,y) = \int_{\R^D} |x_i-(P_i\circ B_{\text{gen}})(x)|^2\,d\pi(x).
\]
As a consequence, for all $\pi\in \Pi(\nu_1,\ldots,\nu_p)$,
\[
\sum_{i=1}^p \lambda_i W_2^2(\nu_i,P_i\#\gamma)\leq \int_{\R^D}
c(x)\,d\pi(x).
\]
This holds for any $\pi \in \Pi(\nu_1,\ldots,\nu_p)$
and thus $~\eqref{ref:GWB}\leq~\eqref{eq:MMOT}.$

Conversely, let $\gamma \in \mathcal{P}_2(\R^d)$ and for all $i \in \llbracket 1,p \rrbracket$, let $\eta_i\in\Pi(\nu_i,P_i\#\gamma)$. By the disintegration theorem, there exist probability kernels $(\eta_i^y)_{y\in \R^{d_i}}$
such that 
\[
d\eta_i(x_i,y) = d\eta_i^y(x_i) d(P_i\#\gamma)(y).
\]
For all $f$ positive and measurable function on $\R^{d_i} \times \R^{d_i}$, it thus holds
\begin{align*}
\int_{\R^{d_i}\times \R^{d_i}} f(x_i,y)\,d\eta_i(x_i,y)&=\int_{\R^{d_i}} \left(\int_{\R^{d_i}} f(x_i,y)\,d\eta_i^y(x_i)\right)d(P_i\#\gamma)(y)\\
&=\int_{\R^d} \left(\int_{\R^{d_i}}f(x_i,P_i(y))\,d\eta_i^{P_i(y)}(x_i)\right)d\gamma(y)\\
  &=\int_{\R^d}\left(\int_{\R^D} f(x_i,P_i(y))\,d\eta_1^{P_1(y)}(x_1)\ldots d\eta_p^{P_p(y)}(x_p)\right)\,d\gamma(y).
\end{align*}
Let us then define a probability measure $\pi$ on $\R^D$ by setting for all positive measurable function $g$ on $\R^D$
\[
\int_{\R^D} g(x)\,d\pi(x)=\int_{\R^d}\left(\int_{\R^D} g(x)\,d\eta_1^{P_1(y)}(x_1)\ldots d\eta_p^{P_p(y)}(x_p)\right) d\gamma(y).
\]
By construction, $\pi\in\Pi(\nu_1,\ldots,\nu_p)$. Indeed, for any positive measurable function $g$ on $\R^{d_i}$, it holds
\begin{align*}
\int_{\R^D} g(x_i)\,d\pi(x)&=\int_{\R^d}\left(\int_{\R^D} g(x_i)\,d\eta_1^{P_1(y)}(x_1)\ldots d\eta_p^{P_p(y)}(x_p)\right) \,d\gamma(y)\\
&=\int_{\R^{d_i} \times \R^{d_i}} g(x_i)\,d\eta_i(x_i,y)=\int_{\R^{d_i}} g(x)\,d\nu_i(x).
\end{align*}
Finally, for any distribution $\gamma$ and $\eta_1,\ldots,\eta_p$ in
$\Pi(\nu_1,P_1\#\gamma),\ldots,\Pi(\nu_p,P_p\#\gamma)$, we have 
\begin{align*}
\sum_{i=1}^p \lambda_i \int_{\R^{d_i}\times\R^{d_i}} |x_i-y_i|^2\,d\eta_i(x_i,y_i) &= \sum_{i=1}^p \lambda_i \int_{\R^d} \left( \int_{\R^D} |x_i-P_i(y)|^2) \,d\eta_1^{P_1(y)}(x_1)\ldots d\eta_p^{P_p(y)}(x_p) \right)d\gamma(y)\\
&= \int_{\Rd\times\R^D}\sum_{i=1}^p \lambda_i |x_i-P_i(y)|^2 \,d\eta_1^{P_1(y)}(x_1)\ldots d\eta_p^{P_p(y)}(x_p) d\gamma(y)\\
&\geq \int_{\Rd\times\R^D} \sum_{i=1}^p \lambda_i |x_i-P_i(B_{\text{gen}}(x)))|^2 \,d\eta_1^{P_1(y)}(x_1)\ldots d\eta_p^{P_p(y)}(x_p) d\gamma(y)\\
&=\int_{\Rd} c(x)\,d\pi(x)\geq ~\eqref{eq:MMOT}.
\end{align*}
So we have, for any $\gamma \in \mathcal{P}_2(\R^d)$,  
\[
\sum_{i=1}^p \lambda_iW_2^2(\nu_i,P_i\#\gamma) = \inf_{\eta_1,\ldots,\eta_p} \sum_{i=1}^p \lambda_i \int_{\R^{d_i}\times\R^{d_i}} |x_i-y_i|^2\,d\eta_i(x_i,y) \geq ~\eqref{eq:MMOT}.
\]
It follows that $~\eqref{ref:GWB}\geq~\eqref{eq:MMOT}$ and so $~\eqref{ref:GWB}=~\eqref{eq:MMOT}$.

Finally, if $\pi^*$ is a solution of $~\eqref{eq:MMOT}$, then defining $\gamma^*=B_{\text{gen}}\#\pi^*$, we have
\[
~\eqref{ref:GWB} \leq \sum_{i=1}^p \lambda_i W_2^2(\nu_i,P_i\#\gamma^*) \leq \int_{\R^D} c(x)\,d\pi^*(x) = ~\eqref{eq:MMOT}= ~\eqref{ref:GWB}.
\]
This yields 
\[
~\eqref{ref:GWB}= \sum_{i=1}^p \lambda_i W_2^2 (\nu_i,P_i\#\gamma^*)
\]
and completes the proof.
\end{proof}

Let us give some insight on a specific case where all the probability measures $\nu_i$ are
projections from the same high dimensional probability measure $\nu$.  \begin{proposition}
\label{prop:Psupportedbyaplane}
Assume that $\gamma$ is in $\mathcal{P}_2(\R^d)$ and for each $i$ in
$\llbracket 1,p \rrbracket$, $\nu_i = P_i \# \gamma$. Let $P:\R^d \rightarrow
\R^{d_1+\dots +d_p}$ be the
linear application defined by $P(x)  =  (P_1(x),\dots,P_p(x))$
$\forall x \in \R^d$.   Then $\pi^* = P\# \gamma$ is a solution of ~\eqref{eq:MMOT}. 
If $d < D=d_1+\dots +d_p$, then $\pi^*$ is supported on a subspace of dimension $d$ of $\R^D$. 
\end{proposition}
\begin{proof}
First, observe that the probability measure $\gamma$ is solution of ~\eqref{ref:GWB}. Also, by definition of $B_{\text{gen}}$, we have for each $x \in \R^{d}$,
$B_{\text{gen}}(P(x))  = x$. 
If $\pi^* = P\# \gamma$, then clearly $\pi^*\in \Pi(\nu_1,\ldots,\nu_p)$ and  we have
\begin{equation*}
  \label{eq:Psupportedbyaplane}
\int_{\R^{d_1+\dots+d_p}} \sum_{i=1}^p\lambda_i |x_i - P_i(B_{\text{gen}}(x)|^2\,d\pi^*(x) = 0,
\end{equation*}
which means that $\pi^* = P\# \gamma$ is a solution of ~\eqref{eq:MMOT}. Since $P$ is
linear, and $\gamma$ lives in $\R^d$, if $D>d$ then $\pi^*$ lives in a
subspace of dimension $d$ of $\R^D$.
\end{proof}
For instance, as illustrated by Figure~\ref{fig:Psupportedbyaplane}, for a probability measure $\nu$ on the plane ($d = 2$)
and three linear projections on lines $P_1,P_2, P_3$, then the solution
$\gamma = P\# \nu$ of the multimarginal problem ~\eqref{eq:MMOT} on $\R^3$ will be
supported by a plane.

\begin{figure}
  \centering
  \includegraphics[width=4cm, trim={3cm 2cm 3cm 2.2cm},clip]{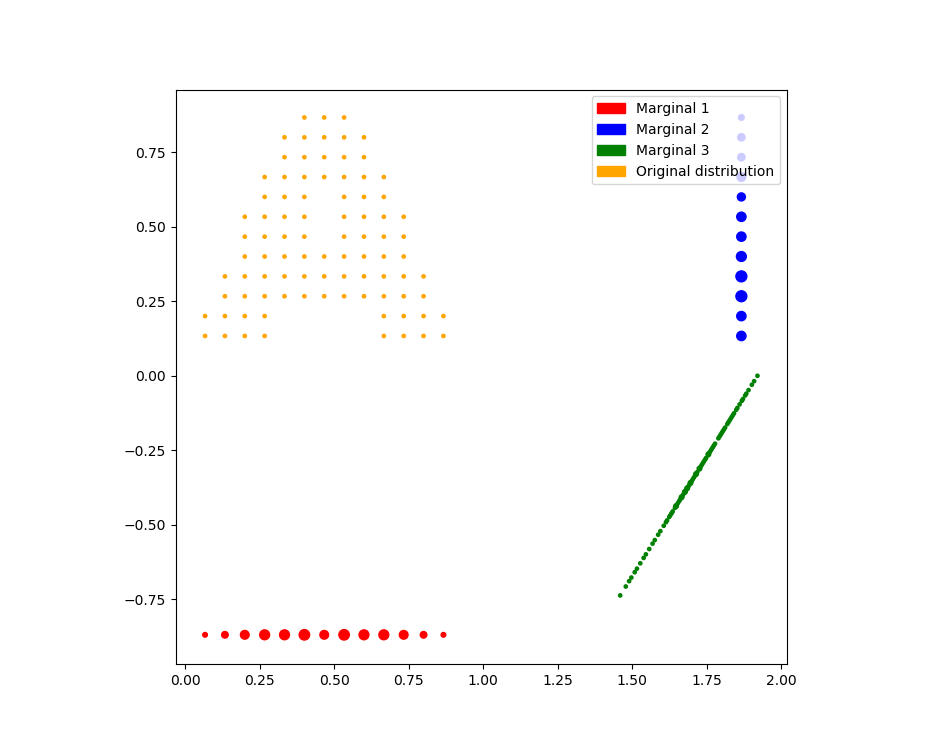}
  \includegraphics[width=4cm, trim={5cm 3cm 2.5cm 2.2cm},clip]{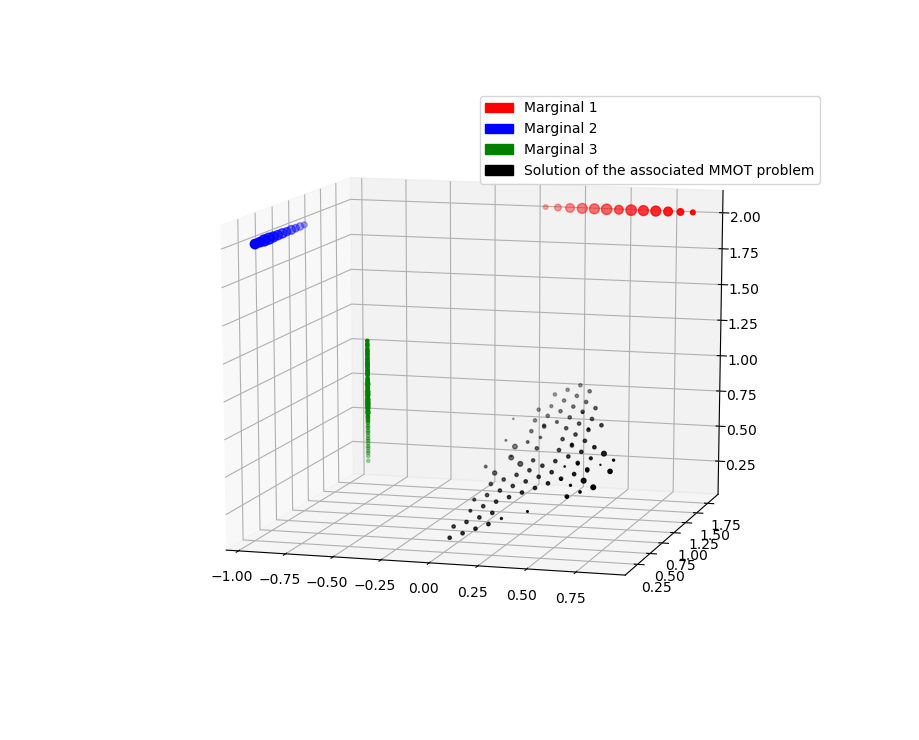}
  \includegraphics[width=4cm, trim={3cm 2cm 2cm 2cm},clip]{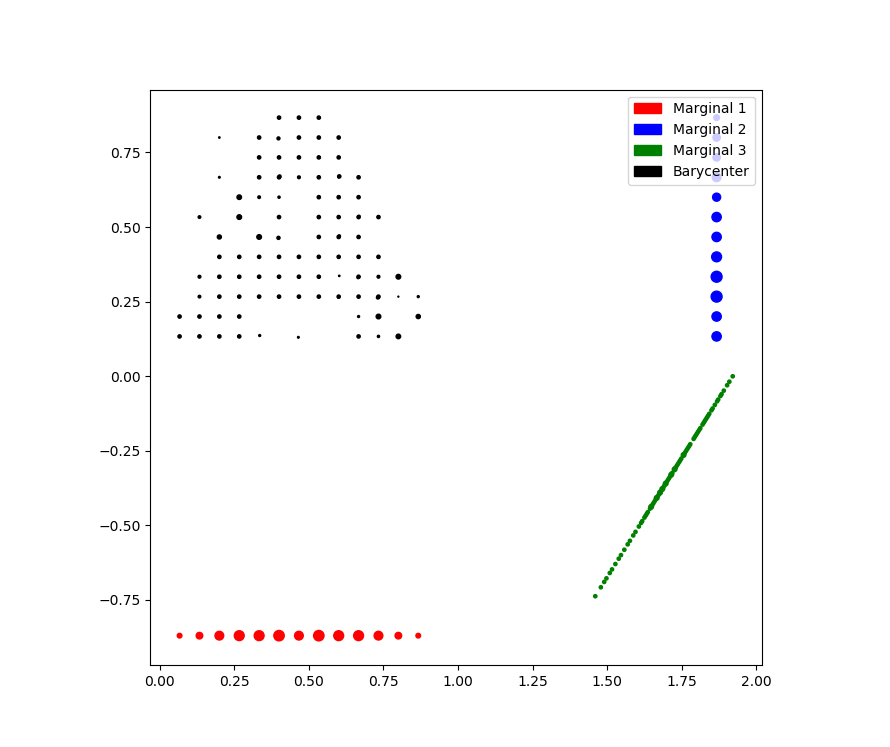}
  
  \caption{Left: A-shaped original distribution (in yellow) and
    three 1D projections. Center: Solution of the corresponding multimarginal
  problem (MM) for these three projections. The solution of the multimarginal problem (MM) is
  supported by a plane, as shown in
  Proposition~\ref{prop:Psupportedbyaplane}. Right: reconstructed Generalized
  barycenter (in black).}
  \label{fig:Psupportedbyaplane}
\end{figure}

\medskip

Proposition \ref{prop:p=GWB} above clarifies the link between the problems~\eqref{ref:GWB} 
and~\eqref{eq:MMOT}. Since~\eqref{eq:MMOT} can be solved by
linear programming, we can derive from this equivalence a way to solve exactly
~\eqref{ref:GWB}. To conclude this section, let us show that it is also possible (and somehow simpler) to relate solutions of \eqref{ref:GWB} to the solutions of the classical multimarginal problem \eqref{eq:W2_multimarge_definition} involving the probability measures $\tilde{\nu}_1,\ldots,\tilde{\nu}_p$ as in Proposition \ref{prop:lien}.

\begin{proposition}
For all $i\in \llbracket 1,p \rrbracket$, let $\tilde{\nu}_i =  (A^{-1/2}P_i^T)\# \nu_i$.
If $\pi^* \in \Pi(\tilde{\nu}_1,\ldots, \tilde{\nu}_p)$ is such that
\[
\int_{(\R^d)^p} |x-B(x)|^2\, d\pi^*(x) = \inf_{\pi \in \Pi(\tilde{\nu}_1,\ldots, \tilde{\nu}_p)} \int |x-B(x)|^2\,d\pi(x),
\]
where $B$ is defined in \eqref{eq:barycenter}, 
then $\gamma^* = A^{-1/2}\#(B\#\pi^*)$ is a solution of \eqref{ref:GWB}.
In other words, if $X_1,\ldots,X_p$ are random vectors whose joint distribution is $\pi^*$, then the law of 
\[
A^{-1/2}\left(\sum_{i=1}^p \lambda_i X_i\right)
\]
is a solution of \eqref{ref:GWB}.
\end{proposition}
\proof
This is a straightforward consequence of Proposition \ref{prop:lien} and Proposition \ref{prop:MM-AC}.
\endproof

\section{Study of the dual problem  \label{sec: dual}}
In what follows, if $F$ is a linear subspace of $\R^n$, we will denote by $\Phi_2(F)$ the set of continuous functions $f:F\to\R$ such that there exists $M\geq0$ such that 
\[
|f|(x) \leq M(1+|x|^2),\qquad \forall x\in F,
\]
where $|\,\cdot\,|$ denotes the standard Euclidean norm on $\R^n.$ We will also consider the set $\Phi_{2,0}(\R^n)$ of continuous functions $f:\R^n \to \R$ such that 
\[
\frac{f(x)}{1+|x|^2} \to 0 \qquad \text{when } |x| \to +\infty.
\]
For all $i \in \llbracket 1,p \rrbracket$, we will denote by $F_i \subset \R^{d_i}$ the range of the linear operator $P_i:\R^d \to \R^{d_i}$ and consider the infimum convolution operator
\[
S_if(x)=\inf_{y\in F_i} \{\lambda_i|x- y|^2-f(y)\},\qquad x\in \R^{d_i}
\]
acting on functions $f:F_i \to \R$. Note that $S_if$ is lower semicontinuous on $\R^{d_i}$ and that $\int S_if(x)\,d\nu(x)$ makes sense in $\R \cup \{-\infty\}$ for any $\nu \in \mathcal{P}_2(\R^{d_i})$.

\begin{theorem}\label{thm:duality}The following duality identity holds
\begin{equation}\label{eq:duality}
\inf_{\gamma \in \mathcal{P}_2(\R^d)} \sum_{i=1}^p \lambda_i W_2^2(\nu_i,P_i\#\gamma) = \sup\left\{\sum_{i=1}^p \int S_ig_i \,d\nu_i : (g_1,\ldots,g_p) \in \prod_{i=1}^p\Phi_2(F_i), \sum_{i=1}^p g_i\circ P_i \geq0\right\}.
\end{equation}
Moreover, if at least one of the $F_i$ has dimension $d$, then
\begin{equation}\label{eq:dualitybis}
\inf_{\gamma \in \mathcal{P}_2(\R^d)} \sum_{i=1}^p \lambda_i W_2^2(\nu_i,P_i\#\gamma) = \sup\left\{\sum_{i=1}^p \int S_ig_i \,d\nu_i : (g_1,\ldots,g_p) \in \prod_{i=1}^p\Phi_2(F_i), \sum_{i=1}^p g_i\circ P_i =0\right\}.
\end{equation}
\end{theorem}
Taking $d_1=\cdots=d_p=d$ and $P_1=\cdots=P_p = \mathrm{I}_d$, then \eqref{eq:dualitybis} gives back Agueh and Carlier dual formulation of the classical Wasserstein barycenter problem (see \cite[Proposition 2.2]{agueh11}). On the other hand, if $d_1=\cdots=d_{p-1}=1$, $d_p=d$ and $P_p=\mathrm{I}_d$, then \eqref{eq:dualitybis} gives back \cite[Theorem 2.3]{abraham17}.

\proof[Proof of Theorem \ref{thm:duality}.]
First let us show how to deduce \eqref{eq:dualitybis} from \eqref{eq:duality}. Without loss of generality one can assume that $F_1$ has dimension $d$. This implies of course that $d_1 \geq d$ and that $P_1:\R^d \to F_1$ is a bijection.  Take $(g_1,\ldots,g_p) \in \prod_{i=1}^p\Phi_2(F_i)$ such that $g:=\sum_{i=1}^p g_i\circ P_i \geq0$, and define $\bar{g}_1 = g_1 - g\circ P_1^{-1}$ and $\bar{g}_i=g_i$ for all $i\geq2$. Then, it is clear that $(\bar g_1,\ldots, \bar g_p) \in \prod_{i=1}^p\Phi_2(F_i)$ and $\sum_{i=1}^p \bar{g}_i\circ P_i=0$. Also, for all $x\in F_1$,
\[
S_1\bar{g}_1(x) = \inf_{y\in F_1} \{\lambda_1|x- y|^2-g_1(y)+g(P_1^{-1}y)\} \geq S_1g_1(x).
\]
From this follows that the right hand-side of \eqref{eq:dualitybis} is greater than or equal to the right hand-side of \eqref{eq:duality}. Since the other inequality is obvious, this completes the proof of \eqref{eq:dualitybis}.

Now, let us prove \eqref{eq:duality}.
First note that, if $(g_1,\ldots,g_p) \in \prod_{i=1}^p\Phi_2(F_i)$ are such that $\sum_{i=1}^p g_i(P_ix) \geq0$ for all $x\in \R^d$, then for any $\gamma \in \mathcal{P}_2(\R^d)$, it follows from (the easy sense of) Kantorovich duality formula that
\[
\sum_{i=1}^p \lambda_i W_2^2(\nu_i,P_i\#\gamma) \geq \sum_{i=1}^p\left( \int S_ig_i \,d\nu_i + \int g_i(P_i)\,d\gamma\right) \geq \sum_{i=1}^p \int S_ig_i \,d\nu_i.
\]
So, we get the following bound
\[
\inf_{\gamma \in \mathcal{P}_2(\R^d)} \sum_{i=1}^p \lambda_i W_2^2(\nu_i,P_i\#\gamma) \geq \sup\left\{\sum_{i=1}^p \int S_ig_i \,d\nu_i : (g_1,\ldots,g_p) \in \prod_{i=1}^p \Phi_2(F_i), \sum_{i=1}^p g_i\circ P_i \geq0\right\}.
\]
Let us now prove the converse inequality.
We consider first the particular case where the matrix $A$ defined in \eqref{eq:matrixA} is invertible. 
Recall the following identity proved in Proposition \ref{prop:lien},
\[
\mathcal{F}(\gamma) = C + \mathcal{G}(A^{1/2}\# \gamma),\qquad \forall \gamma \in \mathcal{P}_2(\R^d),
\]
where $C =  \sum_{i=1}^p \lambda_i \int |x|^2- x^TP_iA^{-1}P_i^Tx d\nu_i(x).$
Since $A^{1/2}$ is invertible, it thus holds
\[
\inf_{\gamma \in \mathcal{P}_2(\R^d)} \mathcal{F}(\gamma) = C+ \inf_{\mu \in \mathcal{P}_2(\R^d)} \mathcal{G}(\mu).
\]
According to the duality result of Agueh and Carlier \cite[Proposition 2.2]{agueh11} applied to the probability measures $\tilde{\nu}_i = (A^{-1/2}P_i^T)\# \nu_i$, $i\in \llbracket 1,p \rrbracket$, it holds
\[
\inf_{\mu \in \mathcal{P}_2(\R^d)} \mathcal{G}(\mu)= \sup\left\{ \sum_{i=1}^p \int S_{\lambda_i} f_i d\tilde{\nu}_i : \sum_{i=1}^p f_i=0, f_1,\ldots, f_p \in \Phi_{2,0}(\R^d)\right\}, 
\]
where, for any $\lambda>0$ and $f:\R^d \to \R$, we set
\[
S_{\lambda} f(x) = \inf_{y\in \R^d} \left\{ \lambda |x-y|^2-f(y)\right\},\qquad x\in \R^d.
\]
Observe that, for all $i \in \llbracket 1,p \rrbracket$ and $f_i \in \Phi_{2,0}(\R^d)$,
\[
\int_{\R^d} S_{\lambda_i} f_i d\tilde{\nu}_i = \int_{\R^{d_i}} S_{\lambda_i} f_i(A^{-1/2}P_i^Tx) \,d\nu_i(x)
\]
and that, for all $x\in \R^{d_i}$,
\begin{align*}
S_{\lambda_i} f_i(A^{-1/2}P_i^Tx)& = \inf_{y\in \R^d} \{\lambda_i |A^{-1/2}P_i^Tx-y|^2-f(y)\}\\
& = \lambda_i( |A^{-1/2}P_i^Tx|^2 -|x|^2) + \inf_{y\in \R^d} \{\lambda_i|x- P_iA^{-1/2}y|^2-\tilde{f}_i(y)\}\\
& = \lambda_i( |A^{-1/2}P_i^Tx|^2 -|x|^2) + \inf_{z\in \R^{d_i}} \{\lambda_i|x- z|^2-g_i(z)\},
\end{align*}
where 
\[
\tilde{f}_i(y) = \lambda_i(|P_iA^{-1/2}y|^2-|y|^2) + f_i(y),\qquad y\in \R^d
\]
and
\[
g_i(z) = \sup \{\tilde{f}_i(y) : P_iA^{-1/2}y = z\},\qquad z\in \R^{d_i}
\]
with the convention $\sup \emptyset = - \infty$.
Note in particular that $g_i(z) = -\infty$ if $z \notin F_i$ and so 
\[
\inf_{z\in \R^{d_i}} \{\lambda_i|x- z|^2-g_i(z)\} = S_ig_i(x),\qquad \forall x\in \R^d
\]
using the same notation $g_i$ for the restriction of $g_i$ to $F_i$.
Therefore, 
\[
\inf_{\gamma \in \mathcal{P}_2(\R^d)} \mathcal{F}(\gamma) = \sup\left\{ \sum_{i=1}^p \int S_{i} g_i d\nu_i : \sum_{i=1}^p f_i=0, f_1,\ldots, f_p \in \Phi_{2,0}(\R^d)\right\}.
\]
Note that
\[
\sum_{i=1}^p g_i(P_ix) \geq \sum_{i=1}^p \tilde{f}_i(A^{1/2}x) = \sum_{i=1}^p  \lambda_i(|P_ix|^2-|A^{1/2}x|^2)+f_i(A^{1/2}x) =0.
\]
Lemma \ref{lem:gi} below shows that for all $i\in \llbracket 1,p \rrbracket$, $g_i \in \Phi_2(F_i)$. Therefore, we get
\[
\inf_{\gamma \in \mathcal{P}_2(\R^d)} \mathcal{F}(\gamma) \leq \sup\left\{\sum_{i=1}^p \int S_ig_i \,d\nu_i : (g_1,\ldots,g_p) \in \prod_{i=1}^p \Phi_2(F_i), \sum_{i=1}^p g_i\circ P_i \geq0\right\},
\]
which gives \eqref{eq:duality}, in the case the matrix $A$ is assumed to be invertible.\\
Let us finally consider the case of a general matrix $A$. According to the discussion at the end of Section \ref{sec:technical} and using the notation introduced there, we know that 
\[
\inf_{\gamma \in \mathcal{P}_2(\R^d)}\mathcal{F}(\gamma) = \inf_{\bar{\gamma} \in \mathcal{P}_2(\R^{\bar{d}})}\bar{\mathcal{F}}(\bar{\gamma}),
\]
where $\bar{d} \leq d$ and $\bar{\mathcal{F}} : \mathcal{P}_2(\R^{\bar{d}}) \to \R^+$ is defined in \eqref{eq:barF}. By construction, the matrix $\bar{A}$ associated to the $\bar{P}_i's$ is invertible. Therefore, using what precedes one concludes that
\begin{align*}
\inf_{\bar{\gamma} \in \mathcal{P}_2(\R^{\bar{d}})}\bar{\mathcal{F}}(\bar{\gamma}) &\leq \sup\left\{\sum_{i=1}^p \int S_ig_i \,d\nu_i : (g_1,\ldots,g_p) \in \prod_{i=1}^p \Phi_2(F_i), \sum_{i=1}^p g_i\circ \bar{P}_i \geq0\right\}\\
& = \sup\left\{\sum_{i=1}^p \int S_ig_i \,d\nu_i : (g_1,\ldots,g_p) \in \prod_{i=1}^p \Phi_2(F_i), \sum_{i=1}^p g_i\circ P_i \geq0\right\},
\end{align*}
where the second line comes from the fact that $g:=\sum_{i=1}^p g_i\circ P_i \geq 0$ if and only if $g$ is non-negative on $K^\perp$, where $K = \bigcap_{i=1}^p \mathrm{Ker}(P_i) = \mathrm{Im}(Q)$ with $Q$ such that $\bar{P}_i = P_iQ$, for all $i\in \llbracket 1,p \rrbracket$. This completes the proof.
\endproof
\begin{lemma}\label{lem:gi}
Let $\lambda>0$ and $f \in \Phi_{2,0}(\R^d)$ and consider the function 
\[
g(z) = \lambda |z|^2 + \sup_{Ry = z} \{f(y) - \lambda |y|^2\},\qquad z\in F,
\]
where $R:\R^d \to \R^{m}$ is some linear map and $F \subset \R^m$ is the image of $R$.
Then $g$ belongs to $\Phi_2(F)$.
\end{lemma}
\proof
Since $f \in \Phi_{2,0}(\R^d)$, $f(y) - \lambda |y|^2 \to - \infty$ as $|y| \to \infty.$ Therefore, there exists $M_1$ such that $f(y) - \lambda |y|^2 \leq M_1$ for all $y\in \R^d$. Thus 
\[
g(z) \leq \lambda |z|^2 + M_1,\qquad \forall z\in F.
\]
On the other hand, denote by $K$ the kernel of $R$ and by $E$ a linear subspace such that $K \oplus E = \R^d.$ The linear map $S : E \to F : x \mapsto Rx$ is then bijective. Let $\|S^{-1}\|$ denote the operator norm of $S^{-1}$ : $\|S^{-1}\| = \sup \{|S^{-1} z|/|z| : z \in F\setminus \{0\}\}.$ For all $z \in F$, it then holds 
\begin{align*}
g(z) \geq \lambda |z|^2 + f(S^{-1}z) - \lambda |S^{-1}z|^2 &\geq -M_2(1+|S^{-1}z|^2) -\lambda |S^{-1}z|^2\\
&\geq -M_2(1+\|S^{-1}\||z|^2) -\lambda \|S^{-1}\||z|^2,
\end{align*}
where $M_2 = \sup \{\frac{|f(x)|}{1+|x|^2} : x \in \R^d\}$. Putting everything together shows that $\frac{|g|}{1+|\,\cdot\,|^2}$ is bounded on $F$. To conclude, let us prove that $g$ is continuous on $F$. Let $z_n$ be some sequence in $F$ converging to some $z \in F$. The sequence $g(z_n)$ is bounded ; define $m = \inf_{n\geq0} g(z_n)$. Since $f(y) - \lambda |y|^2 \to - \infty$ as $|y| \to \infty$, there exists some closed ball $B \subset \R^d$ such that $f(y) - \lambda |y|^2 < m$ when $y \notin B$. Therefore, the supremum defining $g(z_n)$ can be restricted to $B$ :
\[
g(z_n) = \lambda |z_n|^2 + \sup_{Ry = z_n, y\in B} \{f(y) - \lambda |y|^2\}.
\]
Let $\ell = \limsup_{n\to \infty}g(z_n)$. Extracting a sequence if necessary, one can assume without loss of generality that $g(z_n) \to \ell$. 
By compactness of $B$, there exists a sequence $y_n \in B$ such that $Ry_n = z_n$ and
\[
g(z_n) = \lambda |z_n|^2 + f(y_n) - \lambda |y_n|^2.
\]
Again, by compactness of $B$, one can assume without loss of generality that $y_n$ converges to some $\bar{y} \in B$ such that $R\bar{y}=z$. Letting $n\to \infty$, one concludes that
\[
\ell =\limsup_{n\to \infty} g(z_n) = \lambda |z|^2 + f(\bar{y}) - \lambda |\bar{y}|^2 \leq g(z).
\]
On the other hand, if $y^*\in B$ is such that $Ry^*=z$ and $g(z) =  \lambda |z|^2 + f(y^*) - \lambda |y^*|^2$, then defining $a_n = S^{-1}(z_n-z)$, we see that
\[
g(z_n) \geq  \lambda |z_n|^2 + f(y^*+a_n) - \lambda |y^*+a_n|^2.
\]
Since $a_n \to 0$ as $n \to \infty$, we see that
\[
\liminf_{n\to \infty} g(z_n) \geq  \lambda |z|^2 + f(y^*) - \lambda |y^*|^2 = g(z) \geq \ell.
\]
From this, we conclude that $\liminf_{n\to \infty} g(z_n) = \limsup_{n\to \infty} g(z_n)= g(z)$. This shows the continuity of $g$ and completes the proof.
\endproof

\section{Solutions of~\eqref{ref:GWB} for Gaussian distributions}
\label{sec:gaussian_gwb}

In this section, we consider the case where $\nu_1,\ldots,\nu_p$ are Gaussian distributions and we show below that the generalized Wasserstein barycenter can also be a Gaussian distribution, and how its parameters can be computed in practice. 

First let us see how to obtain the expectation of the barycenter. Recall that the quadratic transport cost between two distributions $\eta$ and
$\eta'$ can be written
\[
W_2^2(\eta,\eta') = |\mathbb{E}[\eta]-\mathbb{E}[\eta']|^2+W^2_2\left(\eta-\mathbb{E}[\eta],\eta'-\mathbb{E}[\eta']\right).
\]
This implies that if $\gamma$ is solution of \eqref{ref:GWB} for the
probability measures 
$\nu_1,\ldots,\nu_p$ and weights $\lambda_1,\ldots,\lambda_p$, then
\[
\mathbb{E}[\gamma] = \inf_{x} \sum_i \lambda_i|P_i(x)-\mu_i|^2 = B_{gen}(m_1,\ldots,m_p),
\]
where $m_1,\ldots,m_p$ denote the expectations of the measures $\nu_1,\ldots,\nu_p$.

For this reason, we will assume in all the section that $\nu_1,\ldots,\nu_p$ are centered Gaussian distributions.

\subsection{Existence and characterization of Gaussian solutions of ~\eqref{ref:GWB} for Gaussian distributions}
In this section, we assume that the matrix $A$ defined in \eqref{eq:matrixA} is invertible.

\begin{proposition}\label{prop:existence_Gaussian}
If $\nu_1,\ldots, \nu_p$ are centered Gaussian probability measures, then \eqref{ref:GWB} admits at least one centered Gaussian solution. 
\end{proposition}
\proof
According to Proposition \ref{prop:lien}, it is enough to show that the functional $\mathcal{G}$ admits at least a Gaussian minimizer $\mu^*$. Note that the probability measures $\tilde{\nu}_i$ are (in general degenerate) Gaussian measures as linear transformations of Gaussian measures. To apply Agueh and Carlier results (who consider only non degenerate Gaussian distribution), let us add a small noise to regularize the $\tilde{\nu}_i.$
For all $i\in \llbracket 1,p \rrbracket$, let us thus consider $\tilde{\nu}_{i,n} = \mathrm{Law}(X_i + \frac{1}{n} Z)$, $n\geq 1$, where $X_i \sim \tilde{\nu}_i$ and $Z$ is a standard Gaussian random vector independent of the $X_i$'s. The probability measures $\tilde{\nu}_{1,n},\ldots,\tilde{\nu}_{p,n}$ are Gaussian and absolutely continuous. According to \cite[Theorem 6.1]{agueh11}, the function
\[
\mathcal{G}_n(\mu) =  \sum_{i=1}^p \lambda_i W_2^2(\tilde{\nu}_{i,n}, \mu),\qquad \mu \in  \mathcal{P}_2(\R^d)
\]
attains its minimum at a unique point $\mu_n$, which is a Gaussian probability measure. Fix some measure $\nu_o \in \mathcal{P}_2(\R^d)$. Since $\tilde{\nu}_{i,n} \to \tilde{\nu}_i$ as $n \to \infty$ for the $W_2$ metric, one sees that $\mathcal{G}_n (\nu_o) \to \mathcal{G}(\nu_o)$. In particular, $M := \sup_{n\geq 1}  \mathcal{G}_n(\nu_o) <+\infty$. Since $\mu_n$ is optimal, one gets that $ \mathcal{G}_n (\mu_n) \leq M$ for all $n\geq 1$.
As
\[
\int |x|^2 \,d\mu_n(x) \leq 2W_2^2(\tilde{\nu}_{i,n}, \mu_n) + 2 \int |x|^2\,d\tilde{\nu}_{i,n}
\]
one deduces that
\[
\int |x|^2 \,d\mu_n(x)  \leq 2  \mathcal{G}_n(\mu_n) + 2 \sum_{i=1}^p \lambda_i \int |x|^2\,d\tilde{\nu}_{i,n} \leq 2(M+M') 
\]
where $M' = \sup_{1 \leq i \leq p}\sup_{n\geq 1}  \int |x|^2\,d\tilde{\nu}_{i,n}.$ In particular, we see that the sequence $(\mu_n)_{n\geq 1}$ is tight. Therefore, according to Prokhorov theorem, one can find a sub-sequence $n' \to \infty$ such that $\mu_{n'} \to \mu^*$, as $n' \to \infty.$ Since $\mu^*$ is a limit of Gaussian measures, it is itself Gaussian. Finally, let us show that $\mu^*$ is optimal. For any $\mu$ it holds 
\[
 \sum_{i=1}^p \lambda_i W_2^2(\tilde{\nu}_{i,n'}, \mu_{n'})\leq  \sum_{i=1}^p \lambda_i W_2^2(\tilde{\nu}_{i,n'}, \mu)
\]
so, letting $n' \to \infty$, it follows from the lower semicontinuity of $W_2$ that
\[
 \sum_{i=1}^p \lambda_i W_2^2(\tilde{\nu}_i, \mu^*)\leq  \sum_{i=1}^p \lambda_i W_2^2(\tilde{\nu}_i, \mu)
\]
this completes the proof.
\endproof

\begin{proposition}
If for all $i \in \llbracket 1,p \rrbracket$, $\nu_i$ is a centered Gaussian probability measure on $\R^{d_i}$ with covariance matrix $S_i$, then the \eqref{ref:GWB} admits at least one centered Gaussian solution $\gamma^*$ whose covariance matrix $S$ satisfies the equation
\begin{equation}\label{eq:pointfixe}
A^{1/2}SA^{1/2}=\sum_{i=1}^p \lambda_i \left((A^{1/2}SA^{1/2})^{1/2} A^{-1/2}(P_i^TS_iP_i)A^{-1/2}   (A^{1/2}SA^{1/2})^{1/2} \right)^{1/2}.
\end{equation}
If $S$ is invertible, it also satisfies the following equations
\begin{equation}\label{eq:pointfixe2}
S^{1/2}AS^{1/2}=\sum_{i=1}^p \lambda_i (  S^{1/2} P_i^T S_i P_i S^{1/2} )^{1/2}
\end{equation}
and
\begin{equation}\label{eq:pointfixe3}
S^{1/2} = A^{-1/2} \left( A^{1/2}   \sum_{i=1}^p \lambda_i (  S^{1/2} P_i^T S_i P_i S^{1/2} )^{1/2}          A^{1/2}\right)^{1/2} A^{-1/2}.
\end{equation}
\label{prop:gaussian_solution}
\end{proposition}

\proof
Let us follow the proof of the preceding result and denote by $\tilde{K}_{i,n}$ and $K_n$ the covariance matrices of $\tilde{\nu}_{i,n'}$ and $\mu_{n'}$ (converging subsequences). According to \cite[Theorem 6.1]{agueh11}, it holds
\[
K_n=\sum_{i=1}^p \lambda_i \left(K_n^{1/2} (\tilde{K}_{i,n})K_n^{1/2}\right)^{1/2}.
\]
Denoting by $K$ the covariance matrix of $\mu^*$ and $\tilde{K}_i$ the covariance matrix of $\tilde{\nu}_i$, $i\in \llbracket 1,p \rrbracket$, one knows that $K_n \to K$ and  $\tilde{K}_{i,n} \to \tilde{K}_i = A^{-1/2}P_i^TS_iP_iA^{-1/2}$. Since the map $M \mapsto M^{1/2}$ is continuous on the space of semi-definite symmetric matrices, one concludes that $K$ satisfies the equation
\[
K=\sum_{i=1}^p \lambda_i \left(K^{1/2} (\tilde{K}_{i})K^{1/2}\right)^{1/2}.
\]
Since $S = A^{-1/2}KA^{-1/2}$, one gets
\[
S= \sum_{i=1}^p \lambda_i A^{-1/2}\left((A^{1/2}SA^{1/2})^{1/2} (A^{-1/2}P_i^TS_iP_iA^{-1/2})(A^{1/2}SA^{1/2})^{1/2}\right)^{1/2}A^{-1/2}
\]
which completes the proof of~\eqref{eq:pointfixe}.

Now, observe that if $S$ is invertible, the linear maps 
\begin{equation}\label{eq:Ti}
T_i = A^{-1/2} S^{-1/2}  (  S^{1/2} P_i^T S_i P_i S^{1/2} )^{1/2} S^{-1/2} A^{-1/2}
\end{equation}
and
\begin{equation}\label{eq:Ri}
R_i = (A^{1/2}SA^{1/2})^{-1/2}\left((A^{1/2}SA^{1/2})^{1/2} A^{-1/2}(P_i^TS_iP_i)A^{-1/2}   (A^{1/2}SA^{1/2})^{1/2} \right)^{1/2}(A^{1/2}SA^{1/2})^{-1/2}
\end{equation}
are both positive semi-definite and such that $T_i K T_i = R_i K R_i = \tilde{K}_i$.
Using Lemma~\ref{lemma:squareroot}, we deduce that $T_i = R_i$. Equation~\eqref{eq:pointfixe2} is then equivalent to Equation~\eqref{eq:pointfixe}, replacing $R_i$ by $T_i$ in the equation.  Equation~\eqref{eq:pointfixe3} is a direct consequence of  Lemma~\ref{lemma:squareroot} applied to Equation~\eqref{eq:pointfixe2}.
\endproof

\begin{lemma}
Let $M_1$ and $M_2$ be two symmetric matrices of the same size, with $M_1$ positive definite and $M_2$ positive semi-definite. Then, the unique positive semi-definite solution of $XM_1X = M_2$ is \[X = M_1^{-\frac{1}{2}}(M_1^{\frac{1}{2}}M_2M_1^{\frac{1}{2}})^{\frac{1}{2}}M_1^{-\frac{1}{2}}.\]
\label{lemma:squareroot}
\end{lemma}
\proof
Since $M_1$ is positive definite, we can write $X = M_1^{-\frac{1}{2}}YM_1^{-\frac{1}{2}}$, with $Y$  positive semi-definite, and thus $YY = M_1^{\frac{1}{2}}M_2M_1^{\frac{1}{2}}$, which is positive semi-definite. This matrix has a unique  positive semi-definite square root  $Y = (M_1^{\frac{1}{2}}M_2M_1^{\frac{1}{2}})^{\frac{1}{2}}$. 
\endproof

\subsection{Local minimizers}
In this section, we assume that the matrix $A$ defined in \eqref{eq:matrixA} is invertible and we show an optimality result for  Gaussian measures whose covariance matrix satisfies~\eqref{eq:pointfixe}. In particular, if this covariance matrix is also invertible, the measure will be solution of~\eqref{ref:GWB}.

We will need the following lemma
\begin{lemma}
Let $M$ be a $d\times d$ matrix, then 
\begin{equation}\label{eq:young}
\frac{1}{2} u^T (M^TM)^{1/2}u + \frac{1}{2} v^T (MM^T)^{1/2}v \geq u^tM^Tv,\qquad \forall u,v \in \R^d.
\end{equation}
\end{lemma}
\proof
Since $(MM^T)^{1/2}$ is symmetric, it is a bijection from $\mathrm{Im} ((MM^T)^{1/2}) = \mathrm{Ker}((MM^T)^{1/2})^\perp$ into itself. 
Note that by construction of the square root of a symmetric semi definite matrix it holds $\mathrm{Ker}((MM^T)^{1/2}) = \mathrm{Ker}(MM^T)$. Furthermore, 
\[
x \in \mathrm{Ker}(MM^T)\Leftrightarrow x^tMM^Tx = 0 \Leftrightarrow |M^Tx| = 0 \Leftrightarrow x\in \mathrm{Ker} (M^T) = \mathrm{Im}(M)^\perp.
\]
Therefore, $\mathrm{Im} ((MM^T)^{1/2}) = \mathrm{Im}(M)$.

Now, fix some $u\in \R^d$ and consider the function $f(v) = u^tM^Tv -  \frac{1}{2} v^T (MM^T)^{1/2}v$, $v\in \R^d$.
Then 
\[
\nabla f(v) = Mu-(MM^T)^{1/2}v = 0  \Leftrightarrow v = (MM^T)^{-1/2} Mu.
\]
Since $f$ is convex, one concludes that 
\[
\sup_{v\in \R^d} f(v) = f((MM^T)^{-1/2} Mu) = \frac{1}{2}u^TM^T(MM^T)^{-1/2}Mu.
\]
Finally, note that
\[
M^T(MM^T)^{-1/2}MM^T(MM^T)^{-1/2}M=M^TM
\]
and so $M^T(MM^T)^{-1/2}M = (M^TM)^{1/2}$ and $\sup_{v\in \R^d} f(v) = \frac{1}{2} u^T (M^TM)^{1/2}u$, which completes the proof.
\endproof

\begin{remark}
Note that one can alternatively prove \eqref{eq:young} first in the case where $M$ is invertible (in this case $(MM^T)^{1/2}$ is also invertible) and then extend the validity of the inequality to general matrices by approximation.
\end{remark}

\begin{proposition}
For all $i \in \llbracket 1,p \rrbracket$, let $\nu_i$ be a centered Gaussian probability measure on $\R^{d_i}$ with covariance matrix $S_i$, and let $\gamma^*$ be a centered Gaussian probability measure on $\R^d$ with a covariance matrix $S$ satisfying Equation~\eqref{eq:pointfixe}. Then $\gamma^*$  minimizes $\mathcal{F}$ over the class of all probability measures $\gamma\in \mathcal{P}_2(\R^d)$ such that $\gamma(F)=1$, where $F$ denotes the support of $\gamma^*.$ In particular, if $S$ is an invertible solution of Equation~\eqref{eq:pointfixe} (or equivalently of Equation~\eqref{eq:pointfixe2}), then $\gamma^*$ is a solution of \eqref{ref:GWB}. 
\label{prop:gaussian_solution2}
\end{proposition}
\proof
For all $i \in \Interval{p}$, write $K_i = A^{-1/2}P_i^TS_iP_iA^{-1/2}$ for the covariance matrices of $\tilde{\nu}_i$ and denote by $\mu^*$ the Gaussian measure with covariance matrix $K = A^{1/2} S A^{1/2}$. 
By assumptions, the positive semi definite matrices  $K$ and $K_i$ are related by the identity
\begin{equation}\label{eq:ptfixeencore}
K = \sum_{i=1}^p \lambda_i  ( K^{1/2} K_i K^{1/2} )^{1/2}.
\end{equation}
The support of $\mu^*$ is $\mathrm{Im}(K).$
It is enough to show that $\mu^*$ minimizes $\mathcal{G}$ over the set of probability measures $\mu \in \mathcal{P}_2(\R^d)$ such that $\mu(\mathrm{Im}(K))=1$.

According to the classical formula for the Wasserstein distance between two Gaussian distributions~\cite{dowson1982frechet}, it holds
\[
W_2^2(\mu^*,\tilde{\nu}_i) = \mathrm{Tr}(K) + \mathrm{Tr}(K_i) - 2 \mathrm{Tr}\left( \left(K^{1/2}K_iK^{1/2}\right)^{1/2}\right).
\]
Recall in particular, that $ \mathrm{Tr}\left( \left(K^{1/2}K_iK^{1/2}\right)^{1/2}\right) =  \mathrm{Tr}\left( \left(K_i^{1/2}KK_i^{1/2}\right)^{1/2}\right)$, by symmetry of the Wasserstein distance.
Therefore, using \eqref{eq:ptfixeencore} one gets
\[
\sum_{i=1}^p\lambda_iW_2^2(\mu^*,\tilde{\nu}_i) = \sum_{i=1}^p\lambda_i \mathrm{Tr}(K_i) -  \mathrm{Tr}\left( K\right).
\]

Now let us construct a couple of Kantorovich potentials between $\mu^*$ and $\tilde{\nu_i}$, for all $i\in \Interval{p}.$ For all $i \in \Interval{p}$, define the quadratic forms
\[
q_i(u) = \frac{1}{2}u^T ( K^{1/2} K_i K^{1/2} )^{1/2}u, u\in \R^d\qquad\text{and}\qquad
r_i(v) = \frac{1}{2}v^T ( K_i^{1/2} K K_i^{1/2} )^{1/2}v, v\in \R^d.
\]
Note that $K^{-1/2}$ and $K_i^{-1/2}$ are well defined on $\mathrm{Im}(K) $ and $\mathrm{Im}(K_i)$ respectively. 
Therefore, for all $i\in \Interval{p}$, the functions 
\[
\phi_i(x) = q_i(K^{-1/2}x),\qquad x\in \mathrm{Im}(K)\qquad\text{and}\qquad \psi_i(y) = r_i(K_i^{-1/2}y), \qquad y\in \mathrm{Im}(K_i)
\]
are also well defined.
We claim that $(\phi_i,\psi_i)$ is a couple of Kantorovich potentials for $\mu^*$ and $\tilde{\nu}_i$. 
First note that applying \eqref{eq:young} with $M=K_i^{1/2}K^{1/2}$, $u=K^{-1/2}x$ and $v=K_i^{-1/2}y$ yields to
\begin{equation}\label{eq:pot1}
\phi_i(x) +\psi_i(y) \geq x\cdot y,\qquad \forall x\in \mathrm{Im}(K), \forall y\in\mathrm{Im}(K_i).
\end{equation}
Moreover, if $u_1,\ldots,u_{k}$ is an orthonormal basis of $\mathrm{Im}(K)$, and $Z_1,\ldots,Z_k$ independent standard Gaussians, then
$\mu^* = \mathrm{Law} (K^{1/2}\sum_{i=1}^k Z_iu_i)$. Therefore,
\[
\int \phi_i(x)\, d\mu^*(x)  = \E\left[\phi_i\left(K^{1/2}(\sum_{i=1}^k Z_iu_i)\right)\right] = \E\left[q_i(\sum_{i=1}^k Z_iu_i)\right].
\]
Let $k=\mathrm{dim} (\mathrm{Im}(K))$ ; note that $( K^{1/2} K_i K^{1/2} )^{1/2}$ leaves stable $\mathrm{Im}(K)$ and vanishes on $\mathrm{Ker}(K)$. Let us choose $u_1,\ldots,u_k$ as an orthonormal basis of eigenvectors of the restriction of $( K^{1/2} K_i K^{1/2} )^{1/2}$ to $\mathrm{Im}(K)$, and denote by $\alpha_1,\ldots, \alpha_k \geq0$ the corresponding eigenvalues.
Then it holds 
\[
\E\left[q_i(\sum_{i=1}^k Z_iu_i)\right] = \frac{1}{2}\E\left[\sum_{i=1}^k \alpha_i Z_i^2\right] = \frac{1}{2}\sum_{i=1}^k \alpha_i = \frac{1}{2}\mathrm{Tr}(( K^{1/2} K_i K^{1/2} )^{1/2}).
\]
Thus 
\[
\int \phi_i(x)\, d\mu^*(x) =\frac{1}{2}\mathrm{Tr}(( K^{1/2} K_i K^{1/2} )^{1/2})
\]
and similarly, 
\[
\int \psi_i(y)\, d\tilde{\nu}_i(y) =\frac{1}{2}\mathrm{Tr}(( K_i^{1/2} K K_i^{1/2} )^{1/2})=\frac{1}{2}\mathrm{Tr}(( K^{1/2} K_i K^{1/2} )^{1/2}).
\]
This shows that
\begin{multline}\label{eq:pot2}
\int \frac{|x|^2}{2} -  \phi_i(x)\, d\mu^*(x) + \int \frac{|y|^2}{2}- \psi_i(y)\, d\tilde{\nu}_i(y) =\\  \frac{1}{2}\mathrm{Tr}(K) + \frac{1}{2}\mathrm{Tr}(K_i) -  \mathrm{Tr}\left( \left(K^{1/2}K_iK^{1/2}\right)^{1/2}\right) = \frac{1}{2}W_2^2(\mu^*,\tilde{\nu}_i).
\end{multline}
According to \eqref{eq:pot1} and \eqref{eq:pot2}, $(\phi_i,\psi_i)$ is a pair of Kantorovich potentials between $\mu^*$ and $\tilde{\nu}_i$.

Now, let $\mu \in \mathcal{P}_2(\mathbb{R}^d)$, be such that $\mu (\mathrm{Im}(K))=1$. For all $i\in \Interval{p}$, the inequality 
\[
\frac{|x - y|^2}{2} \geq \frac{|x|^2}{2}-\phi_i(x) + \frac{|y|^2}{2} - \psi_i(y),\qquad x\in \mathrm{Im}(K), y\in \mathrm{Im}(K_i)
\]
(which follows from \eqref{eq:pot1}) gives after integration that 
\[
\frac{W_2^2(\mu,\tilde{\nu}_i)}{2} \geq \int_{\mathrm{Im}(K)} \left(\frac{|x|^2}{2}-\phi_i(x)\right)d\mu(x)+\int_{\mathrm{Im}(K_i)} \left(\frac{|y|^2}{2}-\psi_i(y)\right)\,d\tilde{\nu}_i(y).
\]
Summing over $i$,
\begin{align*}
\frac{1}{2}  \sum_{i=1}^p  \lambda_iW_2^2(\mu,\tilde{\nu}_i) &\geq  \sum_{i=1}^p  \lambda_i \int_{\mathrm{Im}(K)} \left(\frac{|x|^2}{2}-\phi_i(x)\right)d\mu(x)+\sum_{i=1}^p  \lambda_i\int_{\mathrm{Im}(K_i)} \left(\frac{|y|^2}{2}-\psi_i(y)\right)\,d\tilde{\nu}_i(y)\\
&= \sum_{i=1}^p  \lambda_i\int_{\mathrm{Im}(K_i)} \left(\frac{|y|^2}{2}-\psi_i(y)\right)\,d\tilde{\nu}_i(y)\\
& = \sum_{i=1}^p\lambda_i \mathrm{Tr}(K_i) -  \mathrm{Tr}\left( K\right)\\
&= \frac{1}{2} \sum_{i=1}^p \lambda_i W_2^2(\mu^*,\tilde{\nu}_i),
\end{align*}
where the second line comes from the fact that $\sum_{i=1}^p \lambda_i \phi_i(x) = \frac{|x|^2}{2}$, $x\in \mathrm{Im}(K)$, as easily follows from 
\eqref{eq:ptfixeencore}. This completes the proof.
\endproof

\pagebreak
\subsection{Discussion of the uniqueness question in the Gaussian case}
\subsubsection{A simple example of non-uniqueness}
Consider $\nu_1=\nu_2 = \mathcal{N}(0,1)$, $P_1 : \R^2 \to \R : (x,y) \mapsto x$ and $P_2 : \R^2 \to \R : (x,y) \mapsto y$ and (say) $\lambda_1 = \lambda_2 = 1/2$. Then
\[
P_1 = [1 \ 0], \qquad P_2 = [0 \ 1],\quad
A = \lambda_1 P_1^TP_1 +  \lambda_2 P_2^TP_2 = \frac{1}{2} \left[\begin{array}{cc}  1&0   \\0  & 1 \end{array}\right]
\]
and $\tilde{\nu}_1 = \mathcal{N}\left(0,  \left[\begin{array}{cc}  2&0   \\0  & 0 \end{array}\right] \right)$ and $\tilde{\nu}_2 = \mathcal{N}\left(0,  \left[\begin{array}{cc}  0&0   \\0  & 2 \end{array}\right] \right)$. In other words, denoting by $(e_1,e_2)$ the standard basis of $\R^2$, one has $\tilde{\nu}_i = \mathrm{Law} (\sqrt{2}Z e_i)$, where $Z$ is a standard Gaussian random variable. 

\medskip

\noindent \textbf{Fact 1.} Any coupling between $\tilde{\nu}_1$ and $\tilde{\nu}_2$ is optimal.\\
Indeed, if $(X_1,X_2)$ is a coupling between $\tilde{\nu}_1$ and $\tilde{\nu}_2$, then $X_1$ and $X_2$ are almost surely orthogonal vectors of $\R^2$ and so $X_1\cdot X_2=0$ a.s. Therefore, 
\[
\E[|X_1 - X_2|^2] = \E[|X_1|^2] +\E[|X_2|^2] =4.
\]
Thus $W_2^2(\tilde{\nu}_1,\tilde{\nu}_2)=4$ and all couplings are optimal. 

\medskip

\noindent \textbf{Fact 2.} The class of Wasserstein barycenters of $\tilde{\nu}_1$ and $\tilde{\nu}_2$ for $\lambda_1 = \lambda_2 = 1/2$ consists in all probability distributions of the form :  $\mathrm{Law}\left(\frac{1}{2}X_1 + \frac{1}{2}X_2\right)$, where $(X_1,X_2)$ is an arbitrary coupling between $\tilde{\nu}_1$ and $\tilde{\nu}_2$.\\
Indeed, it is easily checked that if $(X_1,X_2)$ is a coupling between $\tilde{\nu}_1$ and $\tilde{\nu}_2$, then for $i=1,2$ 
\[
\E\left[\left|X_i -\left(\frac{1}{2}X_1 + \frac{1}{2}X_2\right) \right|^2\right] =\frac{1}{4} \E[|X_1 - X_2|^2] = 1
\]
and so, denoting by $\tilde{\nu}_{1.5} = \mathrm{Law}\left(\frac{1}{2}X_1 + \frac{1}{2}X_2\right)$, it holds  
\[
\frac{1}{2}W_2^2(\tilde{\nu}_1,\tilde{\nu}_{1.5})+\frac{1}{2}W_2^2(\tilde{\nu}_{1.5},\tilde{\nu}_{2}) \leq 1.
\]
Moreover, since $(a+b)^2 \leq 2 (a^2+b^2)$, for any $\tilde{\nu} \in \mathcal{P}_2(\R^2)$ it always holds
\[
\frac{1}{2}W_2^2(\tilde{\nu}_1,\tilde{\nu})+\frac{1}{2}W_2^2(\tilde{\nu},\tilde{\nu}_{2}) \geq \frac{1}{4} W_2^2(\tilde{\nu}_1,\tilde{\nu}_2)=1.
\]
This shows that $\tilde{\nu}_{1.5}$ is a Wasserstein barycenters of $\tilde{\nu}_1$ and $\tilde{\nu}_2$.

\medskip

\noindent \textbf{Fact 3.} Among the Gaussian generalized barycenters of $\nu_1$ and $\nu_2$, some have densities and some don't.\\
Indeed, for any $\alpha \in [-1,1]$, choosing $X_1=\sqrt{2}Z_1e_1,X_2=\sqrt{2}Z_2e_2$ with $Z=(Z_1,Z_2)$ a centered two dimensional Gaussian random vector with covariance matrix $K_\alpha = 2 S_\alpha$, where
\[
S_\alpha =   \left[\begin{array}{cc}  1&\alpha   \\\alpha & 1 \end{array}\right]
\]
gives 
\[
\tilde{\nu}_{1.5} = \mathrm{Law} \left(\frac{e_1Z_1+e_2Z_2}{2} \right)=  \mathcal{N}\left(0,  S_\alpha/2 \right),
\]
which yields to the following solution of \eqref{ref:GWB}
\[
\gamma_\alpha^* = \mathcal{N}(0, S_\alpha).
\]
Note that $S_\alpha$ is invertible only when $\alpha \in (-1,1).$ In particular, the equation \eqref{eq:pointfixe} has infinitely many solutions.

\medskip

\noindent \textbf{Fact 4.} 
In this example, Equation \eqref{eq:pointfixe} reads as follows
\begin{equation}\label{eq:exptfixe}
S = \left(S^{1/2}A_1S^{1/2}\right)^{1/2}+  \left(S^{1/2}A_2S^{1/2}\right)^{1/2},
\end{equation}
denoting $A_1 = \left[\begin{array}{cc}  1&0   \\0  & 0 \end{array}\right]$ and $A_2 = \left[\begin{array}{cc}  0&0   \\0  & 1 \end{array}\right]$.
One can easily check that $S=A_1$ is solution of \eqref{eq:exptfixe}. Though $A_1$ is solution of the fixed point equation, the Gaussian measure with covariance matrix $A_1$ is not a solution of \eqref{ref:GWB}.

\subsubsection{Conditions for uniqueness of the Gaussian solution}

The following proposition characterizes the sets of matrices $\{P_i\}_{i \in  \llbracket 1,p\rrbracket}$ which ensure the uniqueness of the Gaussian solution.

\begin{proposition}
\label{prop:gauss_uniqueness}
For all $i \in  \llbracket 1,p\rrbracket$, let $(u_i^k)_{1\le k \le d_i}$ be the set of columns of the matrix $P_i^T$. For $i\in \llbracket 1,p\rrbracket $ and $k,m \in \llbracket 1,d_i\rrbracket$, define the symmetric matrix 
$C_i^{k,m} = (u_i ^ k) (u_i^ m )^T + (u_i ^ m) (u_i^ k )^T$. Assume that the $\nu_{i}$, $1\leq i\leq p$ are all non degenerate centered Gaussian probability measures. Then~\eqref{ref:GWB} has a unique centered Gaussian solution if and only if the linear span of the family   $\{C_i^{k,m}\}_{\stackrel{i \in  \llbracket 1,p\rrbracket}{k,m \in \llbracket 1,d_i\rrbracket}}$ is the set of symmetric matrices on $\R^d$.   
\end{proposition}
\proof
 According to Proposition \ref{prop:uniquenessproj}, two Gaussian solutions $\gamma_1$ and $\gamma_2$ of~\eqref{ref:GWB} are such that $P_i \# \gamma_1 = P_i \# \gamma_2$ for each $i$. 
These solutions are equal if and only the set of equations $P_i S P_i^T  =  S_i$ entirely characterizes the covariance matrix $S$. Now, each of these equations can be rewritten as 
$$(u_i ^ k)^T  S  (u_i^ m ) = S_i (k , m),\;\; \forall k,m \in \llbracket 1,d_i\rrbracket.$$
The left terms of these equations can be seen as scalar products between $S$ and the matrices  $C_i^{k,m}$ (seeing matrices as vectors)
$$ \langle C_i^{k,m} , S \rangle =  2  S_i (k , m),\;\; \forall k,m \in \llbracket 1,d_i\rrbracket.$$
This set of equations entirely characterizes $S$ if and only if the  set of $\{C_i^{k,m}\}$ spans the symmetric matrices of $\R^d$.
\endproof

The following result is a direct consequence of what precedes.
\begin{coro}
Let $u_1,\ldots,u_d$ be a basis of $\R^d$ and consider the linear maps $P_{i,j}:\R \to \R : x \mapsto x\cdot (u_i+u_j)$ for $1\leq i\leq j\leq d$. If $\nu_{i,j}$, $1\leq i\leq j\leq d$, is a family of non degenerate centered Gaussian probability measures on $\R$, then for any family of positive weights $\lambda_{i,j}$, $1\leq i\leq j\leq d$ summing to $1$, the corresponding minimization problem \eqref{ref:GWB} (with $p=d(d+1)/2$) admits a unique centered Gaussian solution. 
\end{coro}

\subsection{Iterative scheme to compute the Gaussian solution}
\label{sec:iterative_gaussian}
In order to compute numerically the solution of the Gaussian problem, we can implement the strategy proposed in~\cite{Alvarez-Esteban:2015aa} for Gaussian barycenters. The main difference is that none of the Gaussian measures $\tilde{\nu}_i =  A^{-1/2} P_i^T \# \mathcal{N}(0,S_i)$ is full rank, which means that the main results of~\cite{Alvarez-Esteban:2015aa} do not apply directly. We show in the following how they can be adapted.

\begin{proposition}
Assume that for each $1\le i\le p$, $S_i$ is a $d_i\times d_i$ symmetric positive definite matrix, and write $\tilde{K_i} = A^{-1/2} (P_i^T S_i P_i) A^{-1/2}$. For some   $d\times d$ definite positive symmetric matrix $K_0$, define the sequence
\begin{equation}
\label{eq:iterative_gauss}
K_{n+1} =  K_n^{-1/2}\left( \sum_{i=1}^p \lambda_{i} (K_n^{1/2} \tilde{K}_iK_n^{1/2}) 
\right)^{2} K_n^{-1/2},\;\;n\ge 0.
\end{equation}
For all $d\times d$ symmetric non-negative definite matrix $M$, write~\footnote{For the sake of simplicity, we re-use the notation $\mathcal{G}$ here as a function of the symmetric non-negative definite matrix $M$ instead of a function of the measure $\mathcal{N}(0,M)$ as defined in Section~\ref{sec:reformulation}. } \begin{equation}\mathcal{G}(M) = \sum_{i=1}^p \lambda_i W_2^2(\mathcal{N}(0,\tilde{K}_i), \mathcal{N}(0,M)).\end{equation}
The iterative scheme~\eqref{eq:iterative_gauss} defines a sequence of   positive definite matrices $K_n$ such that $\mathcal{G}(K_n)$ 
decreases with $n$. Moreover, there exists a subsequence $(K_{n_l})_{n\ge 0}$ converging to a positive semi-definite matrix $K$ and  $S = A^{-1/2} K A^{-1/2}$ satisfies Equation~\eqref{eq:pointfixe}. If $K$ is positive definite, then $\mathcal{N}(0,S)$ is a solution of~\eqref{ref:GWB}.  Furthermore, if the hypotheses of Proposition~\ref{prop:gauss_uniqueness} are satisfied, then $K$ is the unique solution of~\eqref{ref:GWB}.
\end{proposition}
\proof
For any $d\times d$  positive definite symmetric  matrix $K$, define $$L(K)  = K^{-1/2}\left( \sum_{i=1}^p \lambda_{i} (K^{1/2} \tilde{K}_iK^{1/2})^{1/2} 
\right)^{2} K^{-1/2}.$$ Observe that $L(K)$ is still symmetric and positive semidefinite. Since all matrices $S_i$ are positive definite, we can find $\alpha>0$ such that $S_i \ge \alpha \mathrm{I}_d $ for all $i$. It follows that
$$  \alpha \mathrm{I}_d = \alpha \sum_{i=1}^p \lambda_i A^{-1/2}P_i^T  P_i A^{-1/2}\leq \sum_{i=1}^p \lambda_i A^{-1/2}P_i^T S_i P_i A^{-1/2} = \sum_{i=1}^p\lambda_{i}  \tilde{K}_i.$$  
Thus, $\sum_{i=1}^p\lambda_{i}  \tilde{K}_i$ is invertible and so is $\sum_{i=1}^p\lambda_{i}  K^{1/2}\tilde{K}_iK^{1/2}$. Now, choose $\delta$  such that $K^{1/2}\tilde{K}_iK^{1/2} \le \delta \mathrm{I}_d$ for all $i$. Clearly,  
$$ \sum_{i=1}^p\lambda_{i}  (K^{1/2}\tilde{K}_iK^{1/2})^{1/2} \ge \frac{1}{\sqrt{\delta}}\sum_{i=1}^p\lambda_{i}  (K^{1/2}\tilde{K}_iK^{1/2})$$ so $ \sum_{i=1}^p\lambda_{i}  (K^{1/2}\tilde{K}_iK^{1/2})^{1/2} $ remains full rank and so is $L(K)$.
It follows that the iterative scheme~\eqref{eq:iterative_gauss} is well defined and generates a sequence of  positive definite symmetric matrices. 

Now, let $(K_n)_{n\geq0}$ be such a sequence. 
Since $K_n$ is invertible and $K_{n+1}=L(K_n)$, using Proposition 3.3 in~\cite{Alvarez-Esteban:2015aa}, we have
\[
\mathcal{G}(K_n) - \mathcal{G}(K_{n+1}) \ge W_2^2(\mathcal{N}(0,K_n),\mathcal{N}(0,K_{n+1})).\]
The sequence $(\mathcal{G}(K_n))_{n\ge 0}$ is thus positive and decreasing, and so it converges in $\R^+$. This implies in particular that $W_2(\mathcal{N}(0,K_n),\mathcal{N}(0,K_{n+1})) \stackrel{n\rightarrow \infty}{\longrightarrow} 0$. Since $(\mathcal{G}(K_n))_{n\ge 0}$ converges, we can easily deduce that $\mathcal{N}(0,K_n)$ is tight (using the same argument as in the proof of Proposition \ref{prop:existence_Gaussian}) and that the sequence of covariances $(K_n)_{n\ge 0 }$ is bounded and thus has a  subsequence $(K_{n_l})_{l\ge 0 }$  which converges to a symmetric matrix  $K$. This matrix $K$ is positive semidefinite but might not be full rank. It follows that the sequence $\mathcal{N}(0,K_{n_l})$ converges in the space   of Gaussian measures on $\R^d$ equipped with the Wasserstein metric $W_2$ towards $\mathcal{N}(0,K)$. Now,
\[W_2(\mathcal{N}(0,K_{n_l+1}),\mathcal{N}(0,K)) \le W_2(\mathcal{N}(0,K_{n_l+1}),\mathcal{N}(0,K_{n_l})) + W_2(\mathcal{N}(0,K_{n_l}),\mathcal{N}(0,K))\stackrel{n\rightarrow \infty}{\longrightarrow} 0,\]
so $(K_{n_l+1})_{l\ge 0 }$  also converges to  $K$.
Since $K_{n_l+1} = L(K_{n_l})$, we have $(K_{n_l}^{1/2} K_{{n_l}+1} K_{n_l}^{1/2})^{1/2} =  \sum_{i=1}^p \lambda_{i} (K_{n_l}^{1/2} \tilde{K}_iK_{n_l}^{1/2})^{1/2}$ and by continuity of both terms it follows that 
\[K =  \sum_{i=1}^p \lambda_{i} (K^{1/2}\tilde{K}_iK^{1/2})^{1/2}\]
which means that $S = A^{-1/2} K A^{-1/2}$ is a solution of Equation~\eqref{eq:pointfixe}. We deduce that $\mathcal{N}(0,K)$ is a local minimizer of~\eqref{ref:GWB}, as defined in Proposition~\ref{prop:gaussian_solution2}, and that if $K$ is full rank, $\mathcal{N}(0,K)$ is a solution of~\eqref{ref:GWB}.
\endproof

\section{Experiments \label{sec: experiments}}

This section gathers experiments illustrating the behavior of the
generalized Wasserstein barycenters. We start by 
explaining how we can numerically compute solutions of~\eqref{ref:GWB}, using either linear programming to compute the exact solution, or the Sinkhorn algorithm to compute an approximation of the solution.
We then present several results of
generalized barycenters between disagreeing projections in 2 or 3 dimensions, illustrating
how solutions find a compromise between several distributions which
do not coincide on their common subspaces.
The section concludes with experiments on Gaussian distributions.

\subsection{Computing numerical solutions of ~\eqref{ref:GWB}}
\label{sec:computing}

To solve~\eqref{ref:GWB} between empirical distributions $\nu_i$ ({\em i.e.} distributions which can be written as linear combinations of Dirac masses), several strategies are possible. As described in Section~\ref{sec:reformulation}, a probability measure $\gamma^*$ minimizes $\mathcal{F}$ for the measures $\nu_i$  if and only if $A^{1/2}\# \gamma^*$ minimizes $\mathcal{G}$ (see equation~\eqref{eq:G}) between the measures $\tilde{\nu}_i = (A^{-1/2}P_i^T)\# \nu_i$. Thus, any classical Wasserstein barycenter algorithm can be used to minimize $\mathcal{G}$ and deduce solutions for~\eqref{ref:GWB}. 

When $p=2$, the solution can be directly deduced from an optimal plan (for $W_2$) $\pi^*$ between   $\tilde{\nu}_0$ and $\tilde{\nu}_1$. Indeed, if we define \[\forall x,y \in \R^d,\;\;P_t(x,y) = (1-t)x+ty,\]  
then for any $t\in [0,1]$, the probability measure $\tilde{\nu}_t:=P_t\# \pi^*$
is a barycenter between $\tilde{\nu}_0$ and $\tilde{\nu}_1$ for the weights $(1-t,t)$, {\em i.e.} 
\[\nu_t \in \argmin_{\rho}(1-t)W_2^2(\tilde{\nu}_0,\rho)+ tW_2^2(\tilde{\nu}_1,\rho).\]
To compute this optimal plan, we can use any exact or approximate dedicated algorithm, as illustrated by Figure~\ref{fig:diff_lp_sinkhorn}.

\begin{figure}[h!]
\centering
\includegraphics[trim={120 60 60 60}, clip, width=0.49\linewidth]{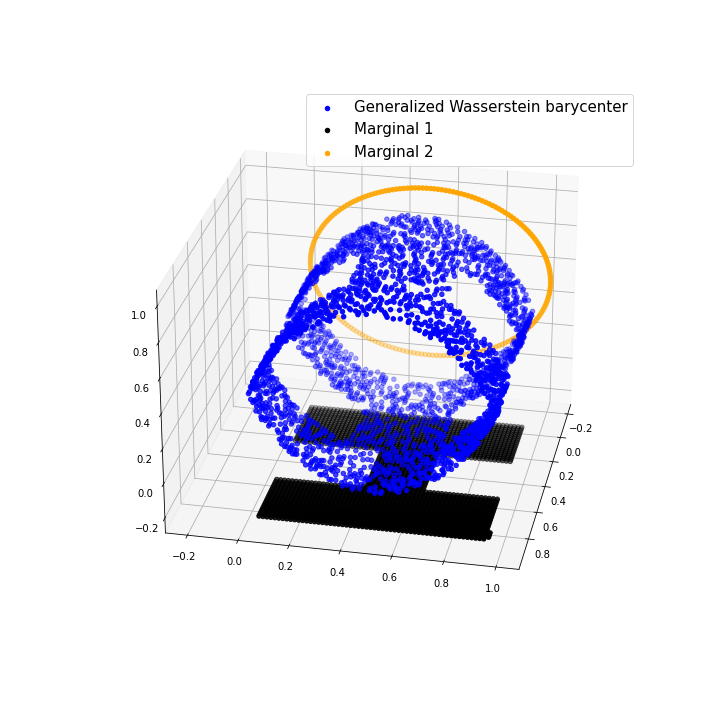}
\hspace{-1em}
\includegraphics[trim={120 60 60 60}, clip,width=0.49\linewidth]{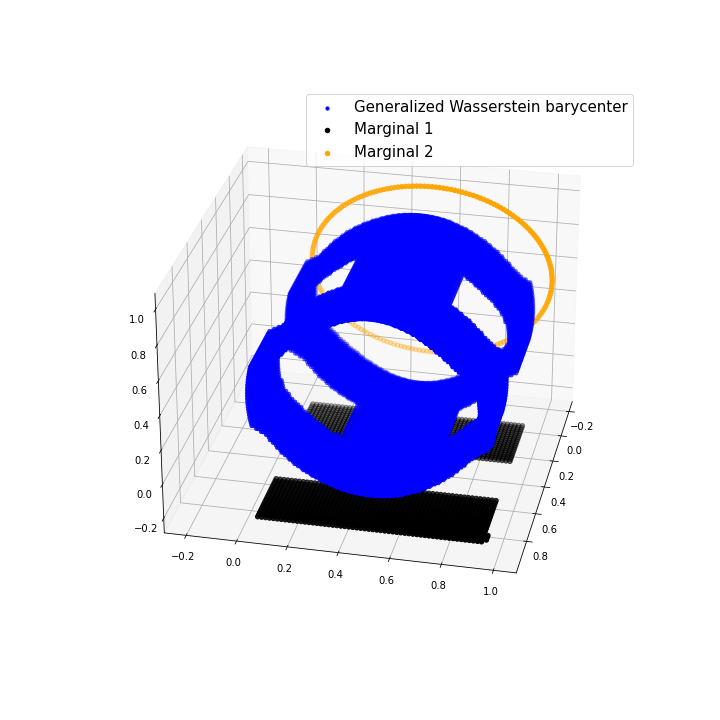}

\caption{Generalized Wasserstein Barycenter (in blue) between two marginals (black 'H' shape and orange circle) $\nu_1$ and $\nu_2$, with $P_1 = \begin{pmatrix}1& 0 & 0 \\ 0  &1 &0 \end{pmatrix}$ and $P_2 = \begin{pmatrix} 0& 1 & 0 \\0  &0 &1 \end{pmatrix}$. 
Left: solution given by solving directly the optimal transport problem between $\tilde{\nu}_1$ and $\tilde{\nu}_2$. Right: approximate solution given by solving the regularized version of the same problem with Sinkhorn algorithm. The resulting solution is thresholded to obtain the displayed set of points.}
\label{fig:diff_lp_sinkhorn}
\end{figure}

When $p \ge 3$, a first possibility is to solve the multimarginal problem~\eqref{eq:MMOT} in order to deduce a generalized barycenter using Proposition~\ref{prop:p=GWB}, as seen in Section \ref{sec: multimarginal}. 
This is possible (although computationally heavy) with any standard linear programming solver if we want to compute the exact solution of the generalized barycenter problem. A faster alternative is to solve the regularized version of this multi-marginal problem with the Sinkhorn algorithm, as described in Section 4.1 of~\cite{benamou15}.    
These approaches using the multimarginal formulation  permit to compute the solution of~\eqref{ref:GWB} when the supports of the distributions $\tilde{\nu}_i$ are not fixed. They minimize $\mathcal{G}$ with respect to both the support and the mass of the distribution. However, their computational cost become prohibitive when the  number of marginals increases and they are therefore limited to a small number of marginals.

\begin{figure}[h!]
\centering
\includegraphics[trim=120 60 60 60, clip, width=0.49\linewidth]{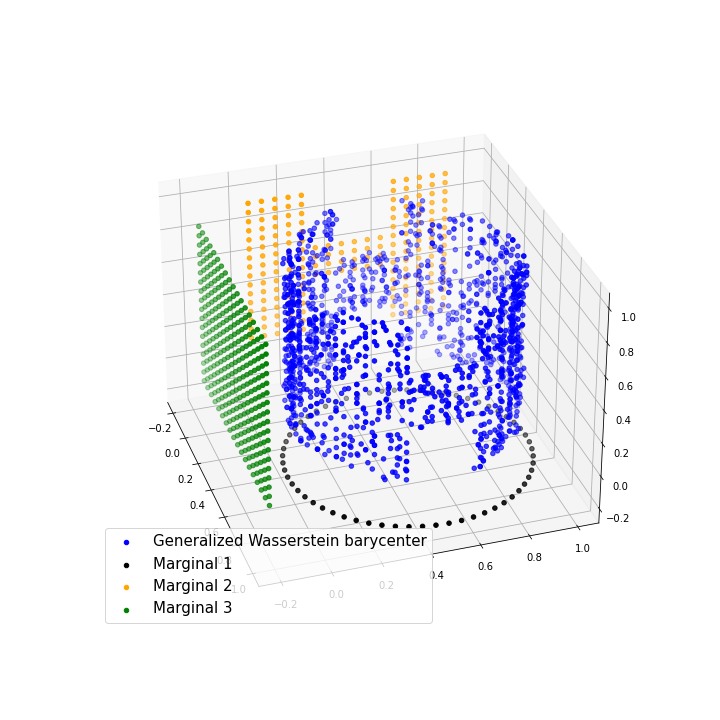}
\hspace{-1em}
\includegraphics[trim=120 60 60 60, clip,width=0.49\linewidth]{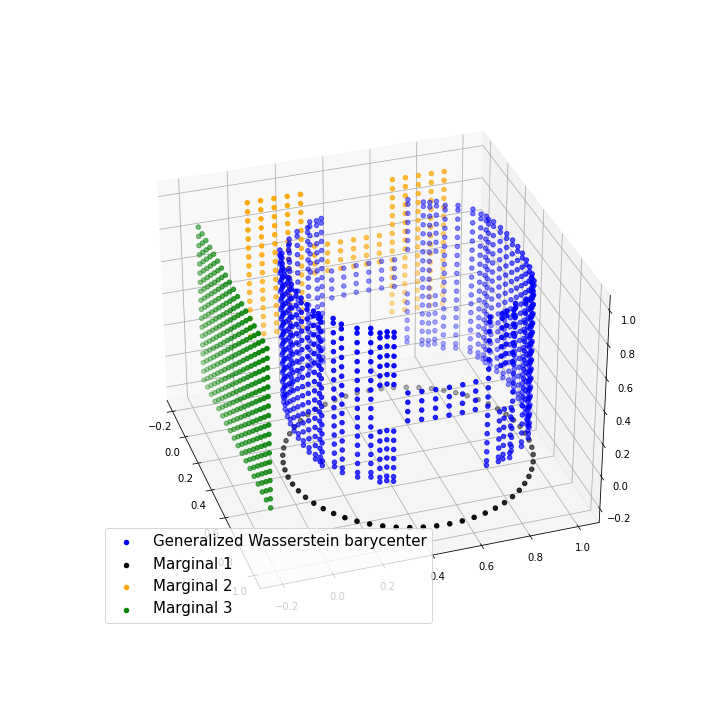}
\caption{Generalized Wasserstein Barycenter (in blue) between three marginals (green square, black circle and orange 'H' shape) $\nu_1$, $\nu_2$ and $\nu_3$, with $P_1 = \begin{pmatrix}
 1& 0 & 0 \\
0  &1 &0  
\end{pmatrix}$, $P_2 = \begin{pmatrix}
 0& 1 & 0 \\
0  &0 &1  
\end{pmatrix}$ and $P_3 = \begin{pmatrix}
 1& 0 & 0 \\
0  &0 &1  
\end{pmatrix}$. Left: barycenter with a fixed number of masses, obtained by optimizing the mass locations but not the weights, as described in~\cite{cuturi2014fast}.
 Right: approximate solution given by solving~\eqref{eq:MMOT}
 with the multimarginal Sinkhorn algorithm and projecting back to find $\gamma^*$ thanks to Proposition~\ref{prop:p=GWB} (the resulting solution is thresholded to obtain the displayed set of points).   }
\label{fig:diff_free_sinkhorn}
\end{figure}

A solution which scales much better with the number of marginals but still belongs to the class of free support approaches is to find a barycenter with a given number of masses, optimizing on the mass locations but not on the weights, as described in~\cite{cuturi2014fast}. This solution provides a fast and convenient way to compute an approximate solution of~\eqref{ref:GWB} when the number of marginals and the number of points in the original measures $\nu_i$ increase. 
Figure~\ref{fig:diff_free_sinkhorn} shows on an example the solutions provided respectively by this approach and by the multimarginal Sinkhorn. Since the solution provided by Sinkhorn contains mass everywhere, it is thresholded to obtain the displayed set of points. The two solutions present similar 3d shapes, although the Sinkhorn solution has a more regular dot pattern, due to the way it is computed. In practice, optimizing only on the mass locations is much faster than the multimarginal strategy.

Finally, for distributions with a known and fixed support (for instance an image grid), iterative Bregman projections that optimize only weights can be used, as described in Section 3.3 of~\cite{benamou2015iterative}.

\subsection{Generalized barycenters between disagreeing marginals }

An interesting aspect of generalized barycenters lies in the way they compromise between disagreeing marginals.
Figures~\ref{fig:disagreeing} and~\ref{fig:disagreeing2} illustrate this behavior on several
examples between different sets of disagreeing
marginals. In these experiments, the solution is computed thanks to the multimarginal Sinkhorn algorithm and projected back thanks to Proposition~\ref{prop:p=GWB}.  For each figure, we show on the left the
 barycenter $\nu$ (black dots) between the original two
dimensional distributions $\nu_i$ (colored dots, each color
corresponding to a different $i$). On the right, we show for each
$i$ the superposition of  $P_i\# \nu$ (black) and $\nu_i$. For
instance, in first experiment, the red square is narrower than
the blue heart in their common dimension, and the
barycenter has to compromise  between these two shapes on this
dimension. 

A last example with three point distributions, representing different superheros, is shown on Figure~\ref{fig:superheros}. The solution is computed here by fixing the number of points to 2000 and optimizing only on their locations, since the number of points in each $\nu_i$ is too large for the multimarginal to run in a reasonable time on a laptop.

\begin{figure}[h!]
        \subfloat[]{%
            \includegraphics[width=0.7\linewidth]{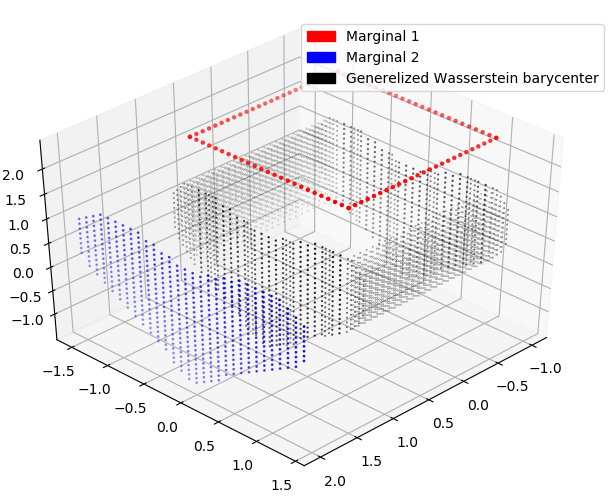}}
        \subfloat[]{\begin{tabular}[b]{c}%
            \includegraphics[width=0.25\linewidth, angle = 90]{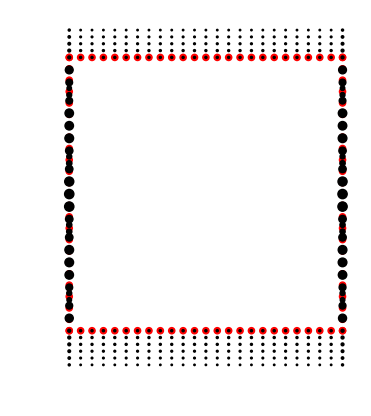}\\            
            \includegraphics[width=0.25\linewidth]{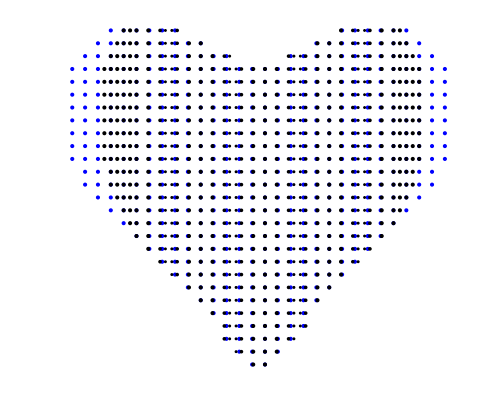}
                    \end{tabular}}

        \subfloat[]{%
            \includegraphics[width=0.7\linewidth]{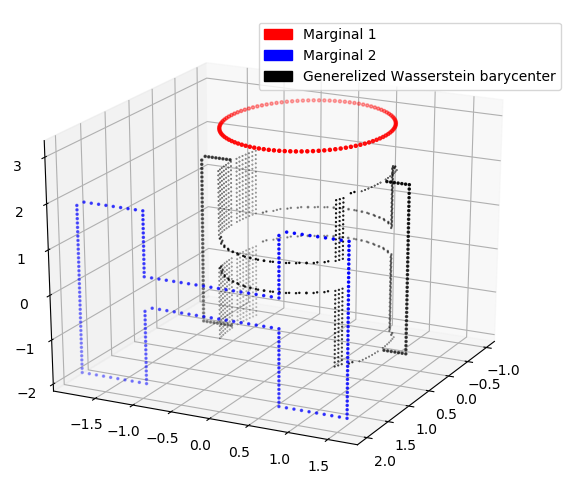}}
        \subfloat[]{\begin{tabular}[b]{c}%
            \includegraphics[width=0.25\linewidth, angle = 90]{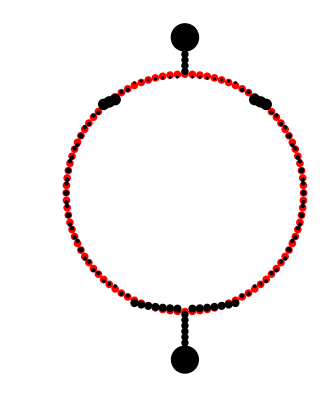}\\            
            \includegraphics[width=0.25\linewidth]{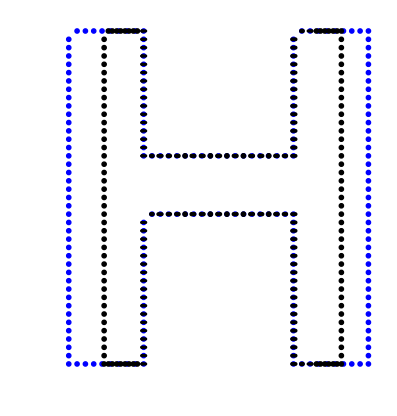}
                    \end{tabular}}
                  
            \caption{Generalized barycenters between disagreeing 2d
      distributions. Each line corresponds to an experiment. On the left, the
three dimensional berycenter $\nu$ (black dots) between the original two
dimensional distributions $\nu_i$ (colored dots, each color
corresponding to a different $i$). On the right, for each
$i$, we show the superposition of  $P_i\# \nu$ (black) and $\nu_i$. }
    \label{fig:disagreeing}        
    \end{figure}

\begin{figure}[h!]
        \subfloat[]{%
            \includegraphics[width=0.7\linewidth]{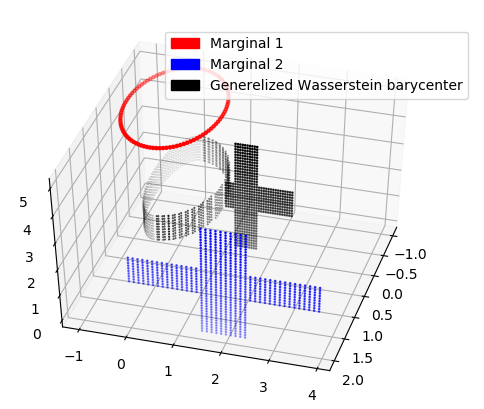}}
        \subfloat[]{\begin{tabular}[b]{c}%
            \includegraphics[width=0.25\linewidth, angle = 90]{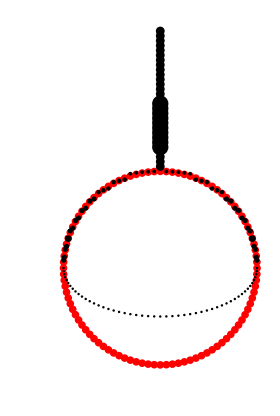}\\            
            \includegraphics[width=0.25\linewidth]{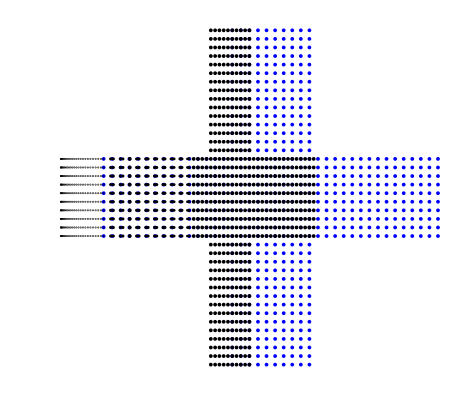}
                    \end{tabular}}

        \subfloat[]{%
            \includegraphics[width=0.7\linewidth]{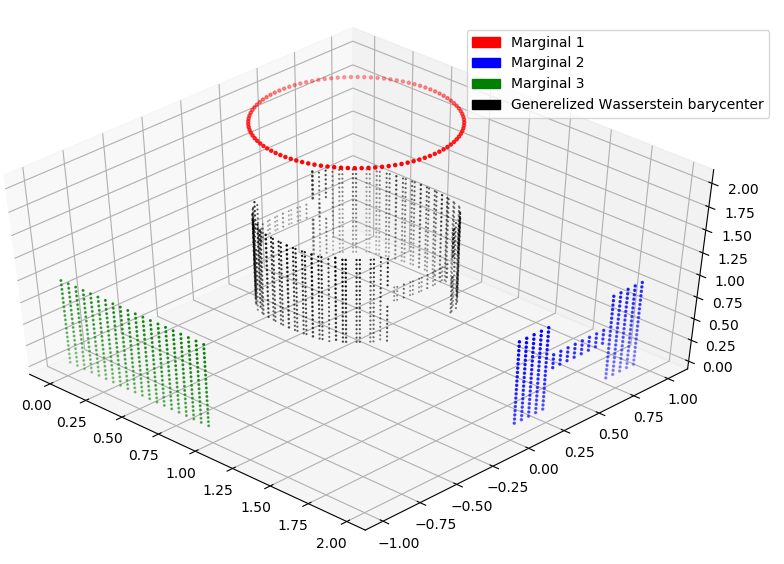}}
        \subfloat[]{\begin{tabular}[b]{c}%
            \includegraphics[width=0.25\linewidth, angle = 90]{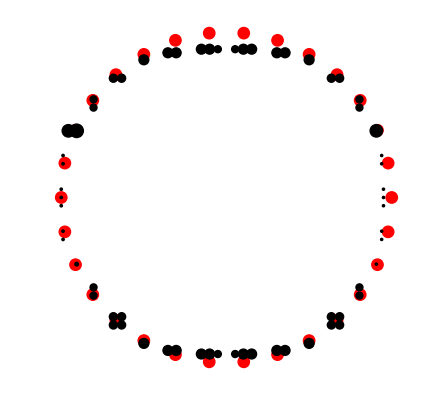}\\            
            \includegraphics[width=0.25\linewidth]{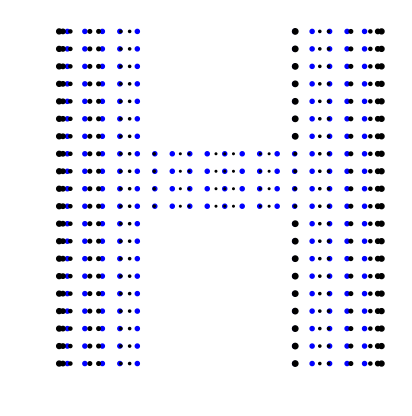}
                    \end{tabular}}
                  
            \caption{Generalized barycenters between disagreeing 2d
      distributions. Each line corresponds to an experiment. On the left, the
three dimensional berycenter $\nu$ (black dots) between the original two
dimensional distributions $\nu_i$ (colored dots, each color
corresponding to a different $i$). On the right, for each
$i$, we show the superposition of  $P_i\# \nu$ (black) and $\nu_i$. }
    \label{fig:disagreeing2}        
    \end{figure}

\begin{figure}[h!]
\centering
\includegraphics[width=0.9\linewidth]{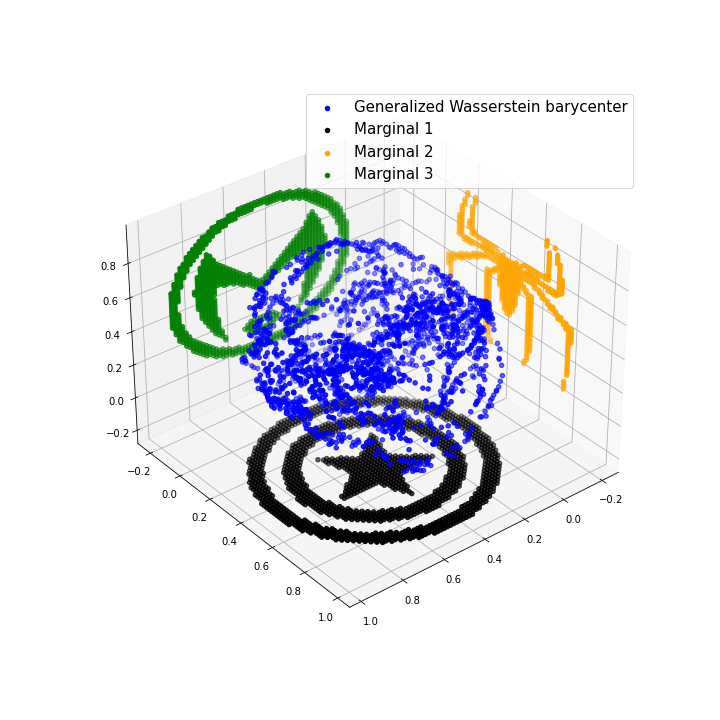}
\vskip -1em

\includegraphics[width=0.8\linewidth]{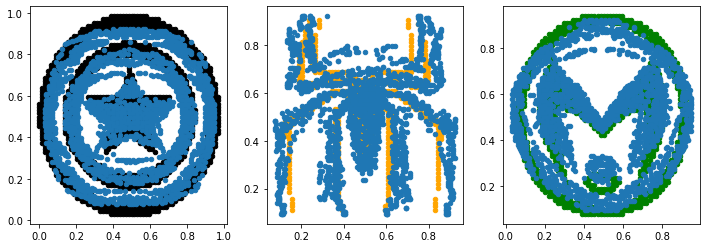}
\caption{Generalized Wasserstein Barycenter (in blue) between three dots distributions representing logos of super heros (Captain America in black, Scarlet Witch in green and Spiderman in orange), with $P_1 = \begin{pmatrix}
 1& 0 & 0 \\
0  &1 &0  
\end{pmatrix}$, $P_2 = \begin{pmatrix}
 0& 1 & 0 \\
0  &0 &1  
\end{pmatrix}$ and $P_3 = \begin{pmatrix}
 1& 0 & 0 \\
0  &0 &1  
\end{pmatrix}$. Top: three marginals and generalized barycenter with a fixed number of masses (2000 points here), obtained by optimizing the mass locations but not the weights, as described in~\cite{cuturi2014fast}.
 Bottom: projections $P_i\# \gamma^*$ of the generalized barycenter superposed with the $\nu_i$. }
\label{fig:superheros}
\end{figure}

\subsection{Generalized Gaussian barycenters}

\subsubsection{Gaussian measures}

We have shown in Section~\ref{sec:gaussian_gwb} how to solve~\ref{ref:GWB} for
Gaussian measures. We have also proven that for such measures, a
finite number of $P_i$ (see Proposition~\ref{prop:gauss_uniqueness}) is needed to reconstruct
the Gaussian perfectly from all the $P_i\#\nu_i$.

We illustrate this property on Figure~\ref{fig:gaussian_measure}, which
shows the reconstruction of a Gaussian measure in 2 dimensions from
its three projections on the 3 axes $(1,0)$, $(0,1)$ and  $
(\frac{1} {\sqrt{2}},   \frac{1} {\sqrt{2}})$,
using the iterative scheme of Section~\ref{sec:iterative_gaussian}
 to reconstruct the covariance matrix. As predicted by
Proposition~\ref{prop:gauss_uniqueness}, the Gaussian measure can be
reconstructed perfectly from these three projections. The same experiment
can be conducted by adding noise to the projections (on the mean and
the standard deviation for instance). In this case, using more
projections increases the robustness of the reconstruction.

\begin{figure}[h!]
  \centering
  \includegraphics[width=0.40\textwidth]{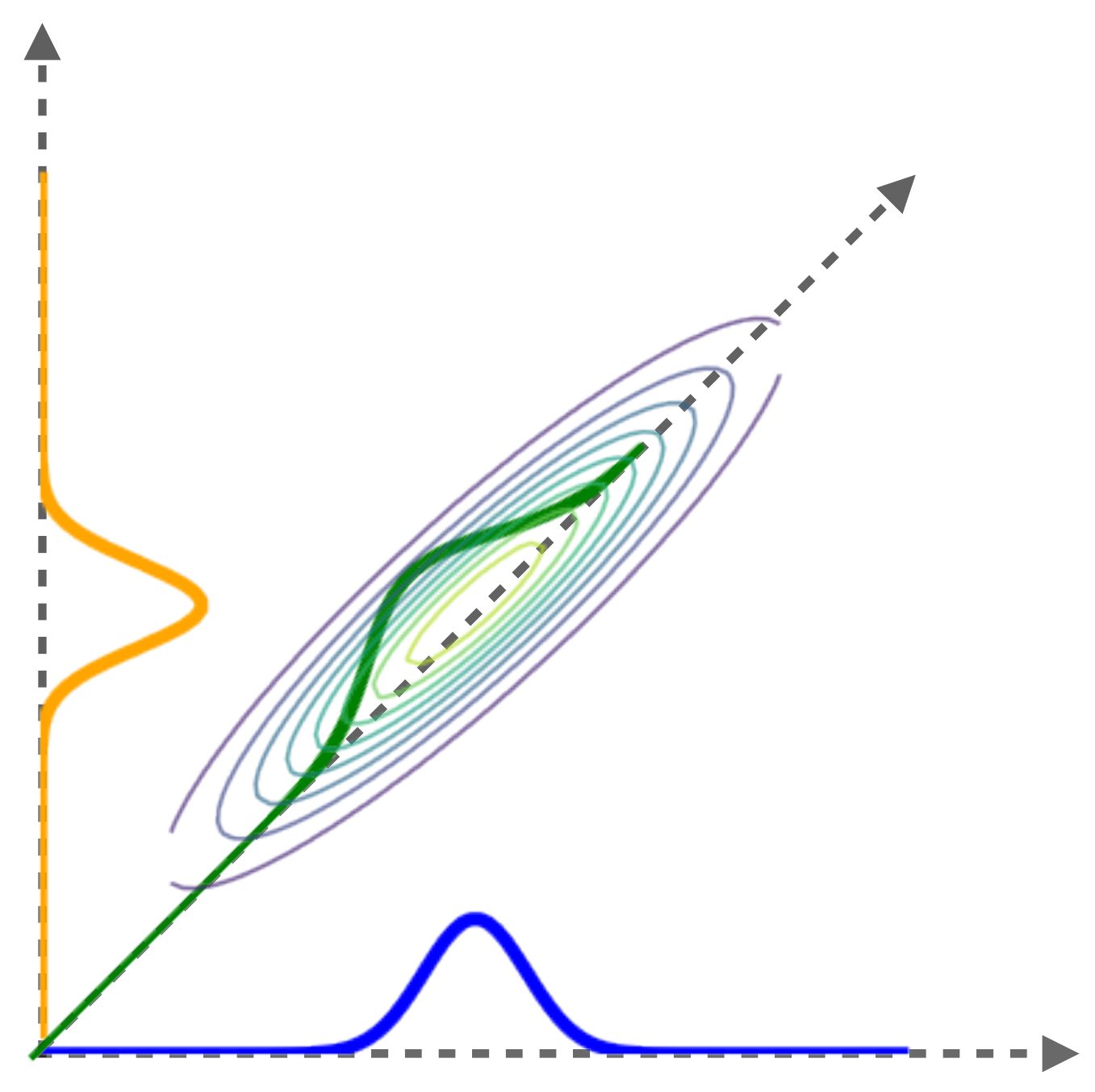}
    \includegraphics[width=0.5\textwidth]{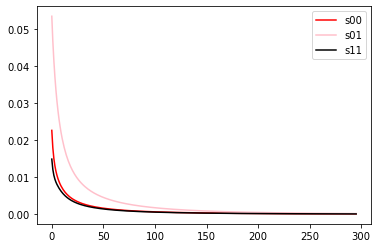}
    \caption{On the left, three one dimensional Gaussian measures
      obtained by projecting the two
      dimensional Gaussian measure centered at $(0.5,0.5)$ with
      covariance matrix 
       $\begin{pmatrix}
 0.06& 0.05  \\
0.05  &0.05  
\end{pmatrix}$
on the three axes $(1,0)$, $(0,1)$ and  $(\frac{1} {\sqrt{2}}, \frac{1} {\sqrt{2}})$. Some level lines of the 2D Gaussian measure  are also displayed on the figure.
On the right, we show the convergence of the different coefficients of the covariance matrix reconstructed from these three projections, using the iterative scheme of Section~\ref{sec:iterative_gaussian}  to reconstruct the covariance matrix.
Each curve represent the absolute difference between a coefficient of the reconstructed covariance matrix and the same coefficient in the original covariance.}
\label{fig:gaussian_measure}
\end{figure}

\subsubsection{Gaussian mixtures}
In~\cite{delon2020wasserstein}, a distance  between Gaussian mixtures on
Euclidean spaces is
defined by restricting the set of possible coupling measures in the
optimal transport problem to Gaussian mixtures. More precisely, if
$\nu_1$ and $\nu_2$ are two Gaussian mixtures on $\R^d$, the distance is defined as
\begin{equation}
  \label{eq:mw2_definition}
   MW_2^2(\nu_1,\nu_2) :=  \inf_{\pi \in \Pi(\nu_1,\nu_2) \cap \mathrm{GMM}_{2d}(\infty)} \int_{\R^d\times\R^d} |x_2-x_1|^2\,d\pi(x_1,x_2),
 \end{equation}
 where $\mathrm{GMM}_{2d}(\infty)$ is the set of all finite Gaussian mixtures
 on $\R^{2d}$.  
A simple discrete
formulation is derived for this distance, which makes it suitable for
problems with massive data, as long as these data are well represented
by Gaussian mixtures.
Barycenters between Gaussian mixtures can be defined for this distance
and  can be
deduced from a corresponding multimarginal problem, as detailed in~\cite{delon2020wasserstein}. This formulation
ensures that barycenters between Gaussian mixtures remain Gaussian
mixtures themselves, which is not the case with classical optimal transport.

We can use this framework to define generalized barycenters between Gaussian mixtures by replacing $W_2^2$ by $MW_2^2$ in~\eqref{ref:GWB}. This makes sense since the $P_i$ are linear, which means that for any Gaussian mixture $\gamma$, all the $P_i\#\gamma$ remain Gaussian mixtures.  More precisely, for $p$ Gaussian mixtures $\nu_1, \nu_2,\dots \nu_p$ on their respective subspaces $\R^{d_i}$, the problem becomes  
\begin{equation}
\inf_{\gamma \in \mathcal{P}_2(\R^d) \cap \mathrm{GMM}_{d}(\infty)} \sum_{i=1}^p \lambda_i MW_2^2(\nu_i,P_i\#\gamma).\label{eq:GMM_projection_consensus}
\end{equation}
The reformulation described in Section~\ref{sec:reformulation} remains valid. Indeed, for two Gaussian mixtures $\nu$ and $\gamma$, and any linear application $T$, 
the equality of Lemma~\ref{lem:couplage} still holds  if we restrict the plans $\pi$ to belong to Gaussian mixtures. As a consequence, we can use the algorithms described in~\cite{delon2020wasserstein} to compute barycenters between Gaussian mixtures in order to solve ~\eqref{eq:GMM_projection_consensus}.

We implemented this approach on a toy example shown on Figure~\ref{fig:GMM}. In this example, a  2D Gaussian mixture $\gamma$ is projected on  three different axes $(1,0)$, $(0,1)$ and  $
(\frac{1} {\sqrt{2}},   \frac{1} {\sqrt{2}})$, giving three 1D Gaussian mixtures $\nu_i$. The corresponding $\tilde{\nu}_i = A^{-1/2}P_i^T\#\nu_i$ are computed and the problem   
\begin{equation}
\inf_{\mu \in \mathcal{P}_2(\R^d) \cap \mathrm{GMM}_{d}(\infty)} \sum_{i=1}^p \lambda_i MW_2^2(\tilde{\nu}_i,\mu)\label{eq:GMM_G_projection_consensus}
\end{equation}
is solved in order to reconstruct the original Gaussian mixture $\gamma$.

\begin{figure}[h!]
  \centering
  \includegraphics[width=0.4\textwidth]{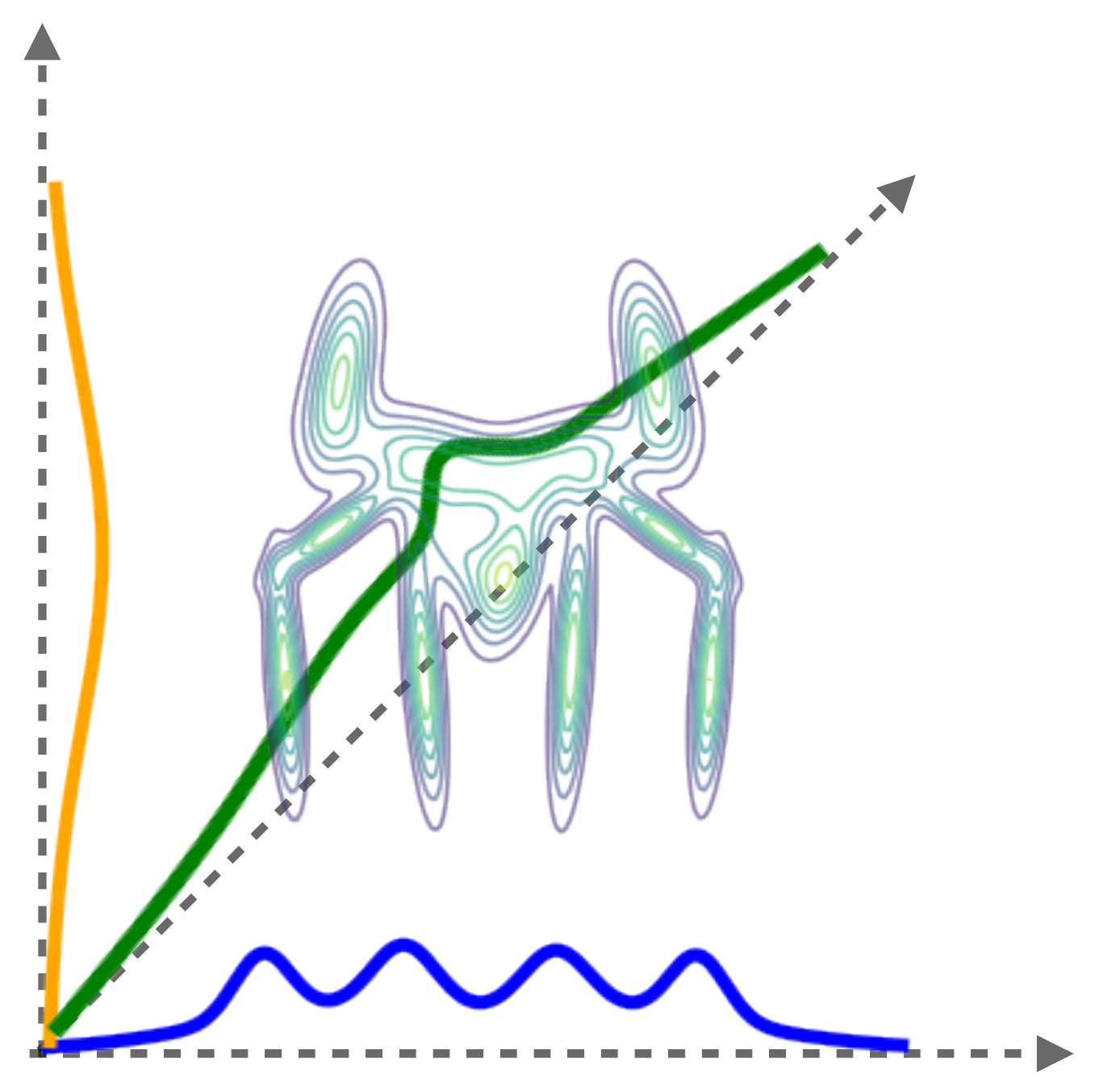}
    \includegraphics[width=0.5\textwidth]{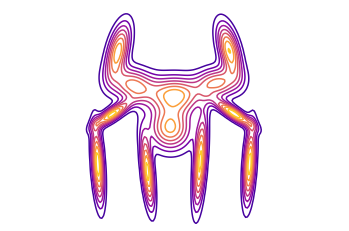}
    \caption{On the left, level lines of a mixture of 12 Gaussians, and projections on the three axes $(1,0)$, $(0,1)$ and  $
(\frac{1} {\sqrt{2}},   \frac{1} {\sqrt{2}})$. On the right, we reconstruct the 2D Gaussian mixture from the simple knowledge of these three one dimensional GMM. }
  \label{fig:GMM}
\end{figure}

\section*{Acknowledgments}
The authors thank Guillaume Carlier for fruitful discussions. 
The first author acknowledges support from the French Research Agency through the MISTIC project (ANR-19-CE40-005) and support from the Institut Universitaire de France.
The second author is supported by a grant of the Simone and Cino Del Duca foundation.  
This research has 
been conducted within the FP2M federation (CNRS FR 2036).

\bibliographystyle{plain}
\addcontentsline{toc}{section}{\refname}
\bibliography{citations.bib}

\end{document}